\documentclass[12pt,a4paper]{amsart}
\usepackage[utf8]{inputenc}
\usepackage{amsfonts}
\usepackage{amssymb}
\usepackage{amsmath}
\usepackage{amsthm}
\usepackage{cite}
\usepackage{enumitem}
\usepackage{tikz}
\usepackage{subfigure}

\theoremstyle{plain}
\newtheorem{thm}{Theorem}
 
\newtheorem{lem}{Lemma}
\newtheorem{prop}{Proposition} 
\newtheorem{cor}{Corollary}
\newtheorem{rem}{Remark}

\newcommand{\vecII}[2]{
\ensuremath{
\begin{pmatrix}
#1 \\ #2 \\
\end{pmatrix}}}
\newcommand{\matII}[4]{
\ensuremath{ 
\begin{pmatrix}  
#1 & #2 \\
#3 & #4 \\
\end{pmatrix}}}

\providecommand{\ind}{1}
\providecommand{\sm}{\setminus}
\providecommand{\N}{\mathbb{N}}
\providecommand{\R}{\mathbb{R}}

\providecommand{\C}{\mathbb{C}}
\providecommand{\cF}{\mathcal{F}}
\newcommand{\Ss}{\mathbb{S}}
\providecommand{\les}{\lesssim}

\providecommand{\eps}{\varepsilon}
\providecommand{\drho}{\:\mathrm{d}\rho}
\providecommand{\dxi}{\:\mathrm{d}\xi}
\providecommand{\dr}{\:\mathrm{d}r}
\providecommand{\dy}{\:\mathrm{d}y}
\providecommand{\dz}{\:\mathrm{d}z}
\providecommand{\ov}{\overline}

\providecommand{\skp}[2]{\langle#1,#2\rangle}
\DeclareMathOperator{\supp}{supp}
\DeclareMathOperator{\sign}{sign}
\DeclareMathOperator{\Real}{Re}
\DeclareMathOperator{\Imag}{Im}
\DeclareMathOperator{\loc}{loc}
\DeclareMathOperator{\ran}{ran}

\renewcommand{\qed}{\hfill $\Box$}

\newcommand{\e}[1]{\mathrm{e}^{#1}}
\renewcommand{\i}{\mathrm{i}}
\newcommand{\norm}[1]{\left\lVert #1 \right\rVert}
\renewcommand{\les}{\lesssim}

\renewcommand{\b}[1]{\textcolor{blue}{#1}}
\renewcommand{\r}[1]{\textcolor{red}{#1}}

\setitemize{itemsep=+2pt}
\setenumerate{itemsep=+2pt}
\def\XXint#1#2#3{{\setbox0=\hbox{$#1{#2#3}{\int}$} 
     \vcenter{\hbox{$#2#3$}}\kern-.5\wd0}}

\begin{document}

\allowdisplaybreaks

\title[Helmholtz equations with a step potential]{An annulus multiplier and applications to the
Limiting absorption principle for Helmholtz equations with a step potential}

\author{Rainer Mandel and Dominic Scheider}
\address{R. Mandel and D. Scheider\hfill\break
Karlsruhe Institute of Technology \hfill\break
Institute for Analysis \hfill\break
Englerstra{\ss}e 2 \hfill\break
D-76131 Karlsruhe, Germany}
\email{rainer.mandel@kit.edu}
\email{dominic.scheider@kit.edu}
\date{\today}


\begin{abstract}
   We consider the Helmholtz equation  $-\Delta u+V \, u - \lambda \, u = f $ on $\R^n$
   where the potential $V:\R^n\to\R$ is constant on each of the half-spaces $\R^{n-1}\times (-\infty,0)$ and  $\R^{n-1}\times (0,\infty)$. We prove an
   $L^p-L^q$-Limiting Absorption Principle for frequencies $\lambda>\max \, V$ with the aid of Fourier
   Restriction Theory and derive the existence of nontrivial solutions of linear and nonlinear Helmholtz
   equations.
   As a main analytical tool we develop new $L^p-L^q$ estimates for a singular Fourier multiplier supported in an annulus.
\end{abstract}

\maketitle

\section{Introduction} 

  In this paper we are interested in the Limiting Absorption Principle (LAP) for the Helmholtz equation on
  $\R^n$ involving a step potential of the form
  \begin{equation}\label{eq:def_V}
    V(x,y) = \begin{cases}
      V_1 & \text{if } x\in\R^{n-1},y>0, \\
      V_2 & \text{if } x\in\R^{n-1},y<0
    \end{cases}
  \end{equation}
  where $V_1\neq V_2$ are two fixed real numbers. We will without loss of generality assume $V_1>V_2$ in the
  following. Examples for elliptic problems involving interfaces modelled
  by potentials of this kind can be found in~\cite[Theorem~1]{DoNaPlRe_Interfaces},
  %(1D NLS below the spectrum)
 \cite[Theorem~2]{DoPlRe_SurfaceGap} %(nD NLS below the spectrum)
 or~\cite{HemKohStaVoi_BoundStates}. %(Spectral theory for dislocation potentials) 
  To explain the motivation behind our study, we recall the interesting phenomenon called
  ``double scattering''. In the context of the Schr\"odinger equation it means that for sufficiently
  regular and fast decaying right hand sides $f$ the unique solution of the initial value problem
  $$
    \i\partial_t \psi - \Delta \psi + V \psi = f \quad\text{in }\R^n,\quad\psi(0)=\psi_0,
  $$
  with $V$ as in~\eqref{eq:def_V}
  splits up into two pieces as $t\to \pm \infty$ that correspond to the two different values of $V$ at
  infinity. This phenomenon is mathematically understood in the one-dimensional case $n=1$
  \cite[Theorem~1.2]{ForVis_DoubleScat}, see also~\cite{DavSim_Scattering,DAncSel_Dispersive}. One byproduct
  of our results is that such a splitting into two pieces may as well be observed for the solutions of the
  corresponding Helmholtz equations in $\R^n$ which are obtained through the Limiting Absorption Principle,
  see for instance the formula~\eqref{eq:formula_u+} where the two parts $f(x,y)1_{(0, \infty)}(\pm y)$ of the right
  hand side contribute differently to the LAP-solution of the Helmholtz equation.
  Notice that solutions $u$ of such Helmholtz equations provide monochromatic solutions
  $\psi(x,t)=\e{\i\lambda t} u(x)$ of the Schr\"odinger equation where $\lambda$
  belongs to the $L^2$-spectrum of the selfadjoint operator $-\Delta+V$ with
  domain $H^2(\R^n)$. We prove our LAP in the topology of Lebesgue spaces in order to treat both linear and
  nonlinear Helmholtz equations. As far as we can see, the more classical  results in weighted $L^2$ spaces resp.
  $B(\R^n),B^*(\R^n)$ (for the definition, cf. \cite[page 4]{AgmonHoer}) by
  Agmon~\cite{Agmon_Spectral,Agmon_Representation,Agmon_ARepresentation} and Agmon-H\"ormander
  \cite{AgmonHoer} do not apply in the nonlinear setting.

 \medskip

 Being interested in the LAP for the Helmholtz operator $-\Delta+V$ we fix the notation
 $$
   \mathcal R(\mu) := (-\Delta+V-\mu)^{-1} \qquad\text{for }\mu\in\C\sm \sigma(-\Delta+V).
 $$
 A computation reveals $\sigma(-\Delta+V)= [\min\{V_1,V_2\},\infty)=[V_2,\infty)$. We aim to prove a LAP
 that is as strong as the corresponding known result for the  constant potential where $V_1=V_2$. In that
 case, $\mathcal R(\mu)$ is bounded from $L^p(\R^n)$ to $L^q(\R^n)$ where $(p,q)\in\mathcal D$ and
 \begin{align*}
   \mathcal D &= \left\{
     (p,q)\in [1,\infty]\times [1,\infty]:\; \frac{1}{p}>\frac{n+1}{2n},\;\frac{1}{q}<\frac{n-1}{2n},\;
     \frac{2}{n+1}\leq \frac{1}{p}-\frac{1}{q}  \leq \frac{2}{n}  
     \right\} \text{ if }n\geq 3, \\
   \mathcal D &= \left\{
     (p,q)\in [1,\infty]\times [1,\infty]:\; \frac{1}{p}>\frac{n+1}{2n},\;\frac{1}{q}<\frac{n-1}{2n},\;
     \frac{2}{n+1}\leq \frac{1}{p}-\frac{1}{q}    < \frac{2}{n}   \right\}
     \text{ if }n=2. 
 \end{align*}
 This is a consequence of results by Kenig-Ruiz-Sogge, Guti\'{e}rrez ($n\geq 3$) and Ev\'{e}quoz ($n=2$) that
 we recall in Theorem~\ref{thm:Gutierrez} along with the bibliographical references. For step potentials of
 the kind~\eqref{eq:def_V} we manage to prove the same result provided that the Restriction Conjecture  is
 true. We refer to Section~\ref{sec:ProofThm} for more information on that topic.
 Given that this conjecture is open for $n\geq 3$, our LAP relies on the best approximation to the Restriction
 Conjecture, which is due to Tao (Theorem~\ref{thm:Tao}). Accordingly, we deal with exponents coming from the set
 \begin{align*}
   \tilde{\mathcal D} 
   &= \left\{
     (p,q)\in [1,\infty]\times [1,\infty]:\; \frac{1}{p}>\frac{1}{p_*(n)},\;\frac{1}{q}<\frac{1}{q_*(n)},\;
    \frac{2}{n+1}\leq \frac{1}{p}-\frac{1}{q} \leq \frac{2}{n}  
     \right\} \text{ if }n\geq 3, \\
   \tilde{\mathcal D} &= \left\{
     (p,q)\in [1,\infty]\times [1,\infty]:\;\frac{1}{p}>\frac{1}{p_*(n)},\;\frac{1}{q}<\frac{1}{q_*(n)},\;
    \frac{2}{n+1}\leq \frac{1}{p}-\frac{1}{q} < \frac{2}{n}  
     \right\} \text{ if }n=2, \\      
     &\text{where}\qquad p_*(n) =\frac{2(n+2)}{n+4}, \quad q_*(n)=\frac{2(n+2)}{n}. 
 \end{align*}
 In particular, $\tilde{\mathcal D}=\mathcal D$ in the case $n=2$  because of
 $p_*(n)=\frac{2n}{n+1}=\frac{4}{3}, q_*(n)=\frac{2n}{n-1}=4$ and  $\tilde{\mathcal
 D}\subsetneq\mathcal D$ in the case $n\geq 3$ because of $p_*(n)<\frac{2n}{n+1},q_*(n)>\frac{2n}{n-1}$.
 Our main result is the following.

 \begin{thm} \label{thm:1}
   Let $n\in\N,n\geq 2$, let $V$ be given by~\eqref{eq:def_V} and $\lambda>V_1>V_2$. Then for all
   $(p,q)\in \tilde{\mathcal D}$ the resolvent estimate
   $$
     \sup_{0<|\eps|\leq 1} \|\mathcal R(\lambda+\i\eps)\|_{L^p(\R^n)\to L^q(\R^n)} <\infty
   $$
   holds. Moreover, the resolvent operators $\mathcal R(\lambda+\i\eps)$ converge  to
   nontrivial operators $\mathcal R(\lambda\pm \i 0)$ as $\pm\eps\searrow 0$  in the weak topology of bounded
   linear operators from $L^p(\R^n)$ to $L^q(\R^n)$. If the Restriction Conjecture is true, then
   the same holds for exponents $(p,q)\in\mathcal D$. 
 \end{thm}

 In particular, our resolvent estimates for $n=2$ coincide with the corresponding estimates for the constant
 potential whereas the ones for $n\geq 3$ cover a smaller range of parameters. For the important class of
 selfdual exponents $p=q'$, however, our Theorem~\ref{thm:1} is optimal since it gives
 $6\leq q<\infty$ for $n=2$ and $\frac{2(n+1)}{n-1}\leq q\leq \frac{2n}{n-2}$ for $n\geq 3$.  
 Let us mention that our result only covers
 frequencies in the range $\lambda>V_1>V_2$ and thus not all frequencies in the (essential) spectrum. We
 believe that   the same estimates can be
 proved for the remaining frequencies $\lambda\in (V_2,V_1]$ in the spectrum with some technical
 work.
 Especially regarding the treatment of Schr\"odinger or wave equations, uniform estimates with respect to all $\lambda\in\C$ would be very helpful and remain a challenging task for the future.

  \medskip

   As an application of Theorem~\ref{thm:1} we consider Helmholtz Equations on $\R^n$
   involving the potential $V$ from~\eqref{eq:def_V}. We start with linear problems of the form
  \begin{equation}\label{eq:nDNLH}
    - \Delta u + V u - \lambda u = f \quad\text{in }\R^n
  \end{equation}
  where $f\in L^p(\R^n)$. Theorem~\ref{thm:1} allows, for $(p, q)\in\tilde{\mathcal D}$,
  to define the outgoing  solution $u_+:=\mathcal R(\lambda+\i 0)f\in L^q(\R^n)$ of this equation. Notice
  that in the context of Helmholtz equations the word ``outgoing'' is used to distinguish $u_+=\mathcal R(\lambda+\i 0)f$
  from the corresponding ``incoming'' solution $u_-:=\mathcal R(\lambda-\i 0)f= \ov{u_+}$,
  see~\cite[Definition~6.5]{AgmonHoer}.
  Combining this with local elliptic regularity theory we obtain the following result.

 \begin{cor} \label{cor:1}
    Let $n\in\N,n\geq 2$, $(p,q)\in\tilde{\mathcal D}$, let $V$ be given by~\eqref{eq:def_V} and
    $\lambda>V_1>V_2$. Then for any $f\in L^p(\R^n)$ the
    Helmholtz equation~\eqref{eq:nDNLH} has a nontrivial ``outgoing'' resp. ``incoming'' strong solution
    $u_+\,(\text{resp. }u_-)\,\in L^q(\R^n)\cap W^{2,p}_{\loc}(\R^n)$ obtained by the Limiting Absorption
    Principle, and there holds an estimate of the form
   $$
     \|u_-\|_{L^q(\R^n)}  + \|u_+\|_{L^q(\R^n)}   \les \|f\|_{L^p(\R^n)}.
   $$
   If the Restriction Conjecture is true, then the same holds for $(p,q)\in\mathcal D$.
 \end{cor}

 Here the symbol $\les$ is used in the sense that there exists some constant $C>0$ depending only on
 the parameters $V_1, V_2, n, p, q,\lambda$ such that $\|u_-\|_{L^q(\R^n)}  + \|u_+\|_{L^q(\R^n)}   \leq C
 \|f\|_{L^p(\R^n)}$. The description of the appropriate radiation conditions
  for ``outgoing'' resp. ``incoming'' solutions remains open and we believe that, here, the results for the
  ranges $\lambda \in (V_2,V_1),\lambda=V_1$ and $\lambda\in (V_1,\infty)$ will be different. Moreover, it
  would be nice to provide a reasonable definition of a Herglotz wave. Recall that in
  the case $V_1=V_2=1$, Herglotz waves are given by $x\mapsto \int_{|\xi|=1} g(\xi)\e{-\i x\cdot\xi}
  \,\mathrm{d}\sigma(\xi)$ for square integrable densities $g$ on the sphere. These
  solutions to the  Helmholtz equation~\eqref{eq:nDNLH} for $f=0$  are of central interest in
  scattering theory.

  \medskip

  In our final result we use the Limiting Absorption Principle from Theorem~\ref{thm:1} to
  prove the existence of solutions to nonlinear Helmholtz equations following the dual variational approach
  developed by Ev\'{e}quoz and Weth~\cite[Theorem~1.2]{EvWe}. We refer to~\cite{Man_LAPperiodic,Man_Uncountably,ManSchei_Dual,EveqWeth_Real,EveqWeth_Branch} for related
  results and other approaches to  nonlinear Helmholtz equations with constant or periodic potentials.

 \begin{cor}\label{cor:2}
   Let $n\in\N,n\geq 2$, let $V$ be given by~\eqref{eq:def_V} and assume $\lambda>V_1>V_2$. Let $\Gamma\in
   L^\infty(\R^n)$ satisfy $\Gamma>0$ on $\R^n$ and $\Gamma(x,y)\to 0$ as
   $|(x,y)|\to\infty$. Then the nonlinear Helmholtz equation
   \begin{align} \label{eq:NLH}
     - \Delta u + V u - \lambda u =   \Gamma  |u|^{q-2}u \quad\text{in }\R^n
   \end{align}
   has a  nontrivial solution in $L^q(\R^n)\cap W^{2,r}_{\loc}(\R^n)$ for all  $r<\infty$  provided
   that $\frac{2(n+1)}{n-1}\leq q< \frac{2n}{n-2}$.
 \end{cor}

 We stress that this result covers the physically relevant special
 cases of the cubic and quintic nonlinearities for $n=3$. More refined dual variational techniques as
 in~\cite{Eveq_Orlicz,Eveq_periodic,EvWe} might be applicable as well to get
 one or even infinitely many solutions for larger classes of nonlinearities.
 For the proof of Corollary~\ref{cor:2} we concentrate on an adaptation of \cite[Theorem~1.2]{EvWe} in order
 to keep the technicalities at a moderate level. Let us mention that the  integrability properties  of the
 solution at infinity are actually slightly better, which can be proved along the lines of
 \cite[Theorem~4.4]{EvWe} with the aid of a bootstrap procedure.

\medskip
  
  In the proof of Theorem~\ref{thm:1} we will use Fourier restriction theory for estimates related to small
  frequencies, whereas our estimates for intermediate frequency ranges require different tools from Harmonic
  Analysis that we believe to be  interesting in themselves. For instance, we encounter linear operators of
  the form
  \begin{equation} \label{eq:def_Tlambdaalpha}
   T_{\lambda,\alpha} h
    :=    \cF_d^{-1}\left( 1_A(\cdot) \e{-\lambda\sqrt{|\cdot|^2-a^2}}
    (|\cdot|^2-a^2)^{-\alpha} m(|\cdot|) \cF_d h(\cdot) \right)
  \end{equation}
  where $\alpha\in \{0,\frac{1}{2}\}$, $m\in C([a,b])$, $\lambda\geq 0$ and $A=\{\xi\in\R^d: a\leq
  |\xi|\leq b\}$ is an annulus with radii $b>a>0$.  The dimensional parameter will be $d=n-1$.
  %This can be seen as a particular annulus
  %multiplier that we investigate as a map from $L^p(\R^d)$ to $L^q(\R^d)$. 
  If $m$ is
  sufficiently smooth and $\lambda=0$, we expect such operators to behave like so-called Bochner-Riesz operators of
  negative order -- a connection that we will highlight below. Their mapping properties are quite well  
  but, as far as we know, not completely understood, especially for $0\leq \alpha<\frac{1}{2}$. In our
  context, however, $m$ is only $\frac{1}{2}$-H\"older continuous and thus we
  cannot build upon on existing literature about these operators. We first present our result dealing with
  the one-dimensional case, which  will be used in the proof of our LAP in the case
  $n=d+1=2$. For completeness, we provide an optimal result under the stonger assumption $m\in
  C^1([a,b])$.
    
  \begin{thm}\label{thm:Tlambdaalpha_d=1}
    Let $\alpha\in [0,1)$, $0<a<b<\infty, \lambda\geq 0$ and $m\in C^1([a,b])$. Then
    $T_{\lambda,\alpha}:L^p(\R)\to L^q(\R)$ is bounded whenever $\frac{1}{p}-\frac{1}{q}\geq \alpha$, $\frac{1}{p}>\alpha$,
    $\frac{1}{q}<1-\alpha$. This range of exponents is optimal under the given conditions and 
    $$
      \|T_{\lambda,\alpha}h\|_{L^q(\R)} \les (1+\lambda)^{2\alpha-\frac{2}{p}+\frac{2}{q}} \|h\|_{L^p(\R)}.
    $$
    For $m\in C([a,b])$ this estimate holds whenever $1\leq p\leq 2\leq q\leq \infty$,
    $\frac{1}{p}-\frac{1}{q}\geq \alpha$.
  \end{thm}
   
  In the higher-dimensional case $d\geq 2$ the matter is more complicated. For our purposes it will be
  sufficient to prove such estimates for exponents $(p,q)$ belonging to the set
  \begin{equation}\label{eq:def_D}
    D_\alpha
    := \left\{ (p,q)\in [1,\infty]^2: \frac{1}{p}>\frac{1}{2}+\frac{\alpha}{2d},\;\,
   \frac{1}{q}<\frac{1}{2}-\frac{\alpha}{2d},\;\, \frac{1}{p}-\frac{1}{q}\geq \frac{2\alpha}{d+1}\right\}
   \quad (0<\alpha<1)
  \end{equation}  
  assuming that the symbol $m$ is continuous. In the case $\alpha=0$ we set $D_0:= [1,2]\times [2,\infty]$. 

\begin{thm} \label{thm:Tlambdaalpha}
   Let $d\in\N,d\geq 2$, $\alpha\in [0,1)$, $0<a<b<\infty, \lambda\geq 0$ and $m\in C([a,b])$. Then
   $T_{\lambda,\alpha}:L^p(\R^d)\to L^q(\R^d)$ is bounded for all $(p,q)\in D_\alpha$ and we have 
   $$
    \|T_{\lambda,\alpha} h\|_{L^q(\R^d)}
    \les (1+\lambda)^\gamma \|h\|_{L^p(\R^d)} \qquad\text{for all }\lambda\geq 0  
  $$  
  for some $\gamma=\gamma_{\alpha,p,q,d}\leq 0$. If additionally
  $\frac{1}{p}-\frac{1}{q}\geq \frac{2}{d+2}$ is assumed then
  \begin{align} \label{eq:gamma_conditions}
    \gamma  \leq 2\alpha-2+\frac{1}{p}-\frac{1}{q} \qquad\text{and}\qquad 
    \gamma  < 2\alpha-2+\frac{1}{p}-\frac{1}{q} \quad\text{if } (\, p= 1 \text{ or } q=\infty\,).
  \end{align} 
\end{thm}

\medskip
 
 \begin{rem} ~
   \begin{itemize}
     \item[(a)] We do not know whether $D_\alpha$ is the optimal range for $\alpha\in (0,1)$ under the given
     assumptions on the symbol $m$. In case $m\in C^1([a,b])$ it is certainly not because boundedness
     also holds whenever  $\frac{1}{p}-\frac{1}{q}\geq \frac{d-1+2\alpha}{2d}$, $\frac{1}{p}>
     \frac{d-1+2\alpha}{2d}$, $\frac{1}{q}< \frac{d+1-2\alpha}{2d}$. 
     In particular, this is true for the exponents $q=\infty,  \frac{d+\alpha}{2d}  \geq \frac{1}{p}>
     \frac{d-1+2\alpha}{2d}$, none of which is covered by Theorem~\ref{thm:Tlambdaalpha}.
     The proof of this result is a straightforward generalization of the proof of
     Theorem~\ref{thm:Tlambdaalpha_d=1} to the higher-dimensional case. The pointwise bound for the kernel $|K_\lambda(z)|\les |z|^{\alpha-1}$
     for $|z|\geq 1+\lambda^2$ from~\eqref{eq:d=1_Klambda_boundII} then generalizes to $|K_\lambda(z)|\les
     |z|^{\alpha-\frac{d+1}{2}}$. 
     \item[(b)] The proof actually yields an explicit expression for the decay rate $\gamma$, which,
     however, might not be optimal. For that reason we only highlight the important
     aspect for us, which is~\eqref{eq:gamma_conditions}. The assumption $\frac{1}{p}-\frac{1}{q}\geq
     \frac{2}{d+2}=\frac{2}{n+1}$ is designed for our application in the context of the LAP from
     Theorem~\ref{thm:1}.
     \item[(c)] The condition $m\in C([a,b])$ can be relaxed to $m\in L^\infty([a,b])$ in the
     non-endpoint case $\frac{1}{p}-\frac{1}{q}>\frac{2\alpha}{d+1}$. Technically, this is due to the fact
     that the operators $\mathcal T_{\lambda,s}$ from~\eqref{eq:def_Tlambdas} are still well-defined for
     $0\leq \Real(s)<1$ (but not for $\Real(s)=1$). Adapting the interpolation procedure from the Proof of
     Theorem~\ref{thm:Tlambdaalpha} accordingly, one obtains the result. Moreover, the assumption
     $m\in C([a,b])$ can be replaced by a continuity assumption near $|\xi|=a$ (keeping the boundedness
     assumption) without changing the result.
   \end{itemize}
 \end{rem}
 
  As indicated above, Theorem~\ref{thm:Tlambdaalpha} can be extended to
  certain exponent pairs  $(p,q)\in [1,\infty]\times [1,\infty]$ not belonging to $D_\alpha$ provided that
  $m$ is sufficiently smooth. One example for such an improvement was given in part (a) of the
  previous remark. The question of optimal ranges of exponents is challenging even in special cases. For
  instance, let us assume $\lambda=0$, $m\equiv 1$ and $A=\{\xi\in\R^d: 1\leq
  |\xi|\leq 2\}$ where $d\geq 2$. Then the operators $T_\alpha:=
  T_{0,\alpha}$  are given by 
  $$
    T_{\alpha} h =  \cF_d^{-1}\left(1_A(\cdot)(|\cdot|^2-1)^{-\alpha}\cF_d h\right). 
  $$
  Their mapping properties are identical to those of so-called Bochner-Riesz operators with
  negative index. The only difference is that the annulus $A$ is replaced by a ball and the singularity
  of the Fourier multiplier at the inner radius of the annulus now occurs at the boundary of the unit ball
  $B\subset\R^d$.
  More precisely, these operators are given by 
  $$
    \tilde T_\alpha h
    :=  \cF_d^{-1}\left( 1_B(\cdot) (1-|\cdot|^2)^{-\alpha}\cF_d h\right).
     %=  const |x|^{\frac{2\alpha-d}{2}} J_{\frac{d-2\alpha}{2}},
   $$
   An alternative description as a convolution operator can be found in
   \cite[p.225]{Boer_estimates}. In the case $\alpha=0$ this is the ball multiplier which is bounded
   from $L^p(\R^d)$ to $L^q(\R^d)$ whenever $1\leq p\leq 2\leq q\leq \infty$. It is known that the operator
   $\tilde T_0$ is bounded on $L^p(\R^d)$ only for $p=2$.
   This is a famous result from 1971 due to Fefferman~\cite{Feff_multiplier}. Up to our knowledge, it is not
   known what the optimal range of exponents for the ball multiplier is.  
   In the case $\alpha\in (0,1)$  the optimal region is contained in the set
   $$
     \mathcal D_\alpha
     := \left\{ (p,q)\in [1,\infty]^2 : \frac{1}{p}-\frac{1}{q} \geq \frac{2\alpha}{d+1},\;
     \frac{1}{p}>\frac{d-1+2\alpha}{2d},\; \frac{1}{q}<\frac{d+1-2\alpha}{2d} \right\},
   $$
   see \cite[Theorem~(iv)]{Boer_estimates}. A standing conjecture is that $\mathcal D_\alpha$ in fact
   coincides with the optimal region, see \cite[Conjecture 2]{KwonLee_Sharp}.    
   In the two-dimensional case $d=2$ the conjecture is true: B\"orjeson \cite[Theorem~(i)]{Boer_estimates}
   proved the corresponding estimates except for the critical line
   $\frac{1}{p}-\frac{1}{q}=\frac{2\alpha}{d+1}$ and the missing piece was proved by Bak~\cite[Theorem~1]{Bak_sharp}. 
   In the higher-dimensional case $d\geq 3$ it is true for    
   $\frac{(d+1)(d-1)}{2(d^2+4d-1)}<\alpha<1$ if $d$ is odd and  
   $\frac{(d+1)(d-2)}{2(d^2+3d-2)}<\alpha<1$ if $d$ is even
  \cite[Theorem 2.13]{KwonLee_Sharp}. We refer to~\cite[Theorem~1.1]{ChoKimLeeShim_Sharp},    
   \cite[Theorem~4]{BakMcMObe_OffTheLine},\cite[Theorem~(iii)]{Boer_estimates} and
   \cite[Theorem~1]{Gut_BochnerRiesz} for earlier results in this direction.
   When restricted to radially symmetric functions $\mathcal D_\alpha$ is optimal in all space
   dimensions and for all $\alpha\in [0,1)$. In the case $\alpha=0$ this follows from~\cite[Theorem
   2]{Herz_OnTheMean} and~\cite{KenTom_TheWeak}, while for $\alpha\in (0,1)$ this result can be found 
   in~\cite[Theorem~1]{BranCol_Bochner}.
   For estimates in the case $\alpha=0$ and $p=q$ with respect to mixed norms we refer to~\cite{Cord_disc}.
   
   \medskip
  
   In addition to Theorem~\ref{thm:Tlambdaalpha} we will need estimates for the 
   operators  $S_{\lambda}:L^p(\R^d)\to L^s(A)$ and their adjoints $S_{\lambda}^*:L^{s'}(A)\to L^{p'}(\R^d)$
   given by
  \begin{align} \label{eq:def_Slambda}
  \begin{aligned}
  S_{\lambda} h
  &:=  1_A(\cdot) \e{-\lambda\sqrt{|\cdot|^2-a^2}} m(|\cdot|)   \cF_d h(\cdot),\\
  S_{\lambda}^* g
  &:=  \cF_d^{-1} \left(  1_A(\cdot)  \e{-\lambda\sqrt{|\cdot|^2-a^2}}  \ov{m(|\cdot|)}  g(\cdot) \right).
  \end{aligned}
  \end{align}
  \begin{thm} \label{thm:Slambdaalpha}
   Let $d\in\N$, $0<a<b<\infty, \lambda\geq 0$ and $m\in L^\infty([a,b])$. Then the 
   operators $S_{\lambda}:L^p(\R^d)\to L^s(A)$, $S_{\lambda}^*:L^{s'}(A)\to
   L^{p'}(\R^d)$ are bounded provided that $1\geq \frac{1}{p}\geq \frac{1}{2},1\geq \frac{1}{s}\geq
   \frac{1}{p'}$ and we have  
    $$
      \|S_\lambda h\|_{L^s(A)} \les  \|h\|_{L^p(\R^d)} (1+\lambda)^{\frac{2}{s'}-\frac{2}{p}-\beta}.
    $$ 
  where $\beta=0$ if $d=1$ and, for sufficiently small $\eps>0$, 
      \begin{align} \label{eq:def_beta}
      \beta = \min\left\{ \frac{d-1}{p}-\frac{d-1}{s'},
    \frac{2(\frac{d+1}{p}-\frac{d-1}{s'}-1)}{p_*(d)'}-\eps,
    \frac{\frac{2}{p_*(d)'}(\frac{1}{p}-\frac{1}{2})}{\frac{1}{p_*(d)}-\frac{1}{2}} -\eps,   
    \frac{2}{p'}\right\} 
      \quad \text{ if }d\geq 2.
    \end{align}
    If the Restriction Conjecture is true, then the same estimate holds for $p_*(d)$ replaced by
    $\frac{2d}{d+1}$.
\end{thm}

%   Da gilt aber
%   $$
%     \beta = \min\left\{ \frac{d-1}{p}-\frac{d-1}{2},
%     \frac{d-1}{p}-\frac{d-1}{2}-\eps,
%     \frac{d-1}{p}-\frac{d-1}{2} -\eps,   
%     \frac{2}{p'}\right\} 
%     = \begin{cases}
%       \frac{d-1}{p}-\frac{d-1}{2} - \eps &, p\geq \frac{2(d+1)}{d+3} \\
%       \frac{2}{p'} &, p<\frac{2(d+1)}{d+3} 
%     \end{cases}
%   $$
  
% \r{
% \begin{thm} \label{thm:Slambda}
%   Let $d\in\N$ and $0\leq \alpha<1$. Then we have the estimate
%   $$
%     \|\cF_d h (|\cdot|^2-1)^{-\alpha}\|_{L^r(A)} \les \|h\|_{L^p(\R^d)}
%   $$
%   provided that  
%   \begin{itemize}
%     \item[(i)] $d=1$ and $1\geq \frac{1}{p}\geq \frac{1+\alpha}{2}$, $1\geq \frac{1}{r}\geq
%     \frac{1}{p'}+\alpha$ or 
%     \item[(ii)] $d=2$ and $1\geq \frac{1}{p}> \frac{d+\alpha}{2d}$, $1\geq \frac{1}{r}\geq \frac{1}{p'} 
%     +\frac{((\alpha+1)p-2)_+}{(d-1)p}$ or
%     \item[(iii)] $d\geq 3$ and $1\geq \frac{1}{p}>\frac{d+2+2\alpha}{2(d+2)}$, $1\geq \frac{1}{r}\geq
%   \frac{1}{p'}+\frac{((\alpha+1)p-2)_+}{(d-1)p}$.   
%   \end{itemize}
%   In the case $\alpha=0$ the endpoint case $p=2$ in (ii),(iii) is included as well.
% \end{thm}
% }
% 
%  \r{The sufficiency part is a direct consequence of Theorem~\ref{thm:Tlambdaalpha_d=1} ($d=1$) and
%   Theorem~\ref{thm:Tlambdaalpha} $(d\geq 2)$. The optimality of the given range of exponents 
%   can be shown with the aid of the Knapp-type counterexample proposed by
%   B\"orjeson in the proof of~\cite[Theorem~(iv)]{Boer_estimates}. In spirit, this proof carries over to our
%   situation so that we dispense with a proof.}

 \medskip

 The outline of this paper is the following. In Section~\ref{sec:nD} we first derive a representation formula 
 for the functions $\mathcal R(\lambda+\i\eps)f,\mathcal R(\lambda\pm\i 0)f$ that are of interest in
 Theorem~\ref{thm:1}. As a main tool we use one-sided Fourier transforms. In Section~\ref{sec:ProofThm} we complete the list of
 required tools from Harmonic Analysis, state all the essential estimates
 (Propositions~\ref{prop:smallfreq},~\ref{prop:interferencefreq},~\ref{prop:intermediatefreq},~\ref{prop:largefreq})
 and combine them in order to prove Theorem~\ref{thm:1}. The application to nonlinear PDEs
 from Corollary~\ref{cor:2} is demonstrated as well. 
 In the Sections~\ref{sec:ProofTlambda_d=1},~\ref{sec:ProofTlambda},~\ref{sec:ProofSlambda}
 we subsequently prove Theorem~\ref{thm:Tlambdaalpha_d=1}, Theorem~\ref{thm:Tlambdaalpha} and
 Theorem~\ref{thm:Slambdaalpha}.
 The Propositions are proved in the following four sections.

 \medskip

 Before starting our analysis let us fix some notation  and conventions. 
 For $p\in [1,\infty]$, we write $L^p(\R^d)$ for the (classical) Lebesgue space of
 complex-valued $p$-integrable functions.  The corresponding
 standard norms are denoted by $\|\cdot\|_{L^p(\R^d)}$. Moreover, we write $p'=\frac{p}{p-1} \in
 [1,\infty]$ for the conjugate exponent. The inner product in $L^2(\R^d)$ is  
 $\skp{f}{g}_{L^2(\R^d)}=\int_{\R^d} f(x)\ov{g(x)}\, \mathrm{d}x$. The $d$-dimensional Fourier transform is
 given by $\cF_dg(\xi) = (2\pi)^{-d/2} \int_{\R^d} g(x) \e{-\i x\cdot\xi}\,\mathrm{d}x$ with inverse
 $\cF_d^{-1} h(\xi) = (\cF_d h)(-\xi)$ where $g,h:\R^d\to\R$ are sufficiently regular. At some points it will
 be convenient to slightly abuse the notation by writing $\mathcal F_d^{-1}(g(\xi))(x)$ in place of $\mathcal
 F_d^{-1}(g)(x)$. The space of complex-valued Schwartz functions is $\mathcal S(\R^n)$.  
 The sphere of radius $\mu$ in $\R^d$ is given by $\Ss^{d-1}_\mu = \{\xi\in\R^d:
 |\xi|=\mu\}$ along with its canonical surface measure $\sigma_\mu$. The corresponding Lebesgue spaces is
 denoted by $L^s(\Ss^{d-1}_\mu),s\in [1,\infty]$.

\section{The representation formula}  \label{sec:nD}

  In this section we derive a representation formula for the  outgoing solution
  of the Helmholtz equation~\eqref{eq:nDNLH} where $V$ is the step potential from~\eqref{eq:def_V}, i.e.,
  $$
    V(x,y) = \begin{cases}
      V_1 & \quad \text{if } x\in\R^{n-1},y>0, \\
      V_2 & \quad \text{if } x\in\R^{n-1},y<0
    \end{cases}
    \qquad\text{with }V_1>V_2.
  $$
  To this end we solve the perturbed Helmholtz equation
  \begin{equation}\label{eq:complexnDNLH2}
    - \Delta u_\eps + V(x,y)u_\eps - (\lambda+\i\eps) u_\eps = f \quad\text{in }\R^n
  \end{equation}
  where $\lambda > V_1 > V_2$ and $\eps > 0$. We define the one-sided
  Fourier transforms of $f\in \mathcal S(\R^n)$ via
  \begin{align*}
   (\cF_n^\pm f)(\xi,\eta) := (\cF_n f_\pm) (\xi,\eta)
   \qquad\text{where }
   f_\pm(x,y):=f(x,y)\cdot 1_{(0, \infty)}(\pm y). 
%     \cF_n^+ f(\xi,\eta)
%     &:= \frac{1}{(2\pi)^{\frac{n}{2}}} \int_{\R^{n-1}\times [0,\infty)} f(x,y)\e{-\i
%     (x,y)\cdot (\xi,\eta)}\, \mathrm{d}(x,y), \\ 
%     \cF_n^- f(\xi,\eta)
%     &:= \frac{1}{(2\pi)^{\frac{n}{2}}} \int_{\R^{n-1}\times (\infty,0]} f(x,y)\e{-\i
%     (x,y)\cdot (\xi,\eta)}\, \mathrm{d}(x,y),
  \end{align*}
%   Notice that this definition makes sense for $\eta\in\C$ with $\Imag(\eta)\leq 0$ respectively
%   $\Imag(\eta)\geq 0$. The generalization to complex arguments will turn out necessary in order to evaluate
%   calculate certain integrals appearing for $\eps>0$ with the aid of the Residue Theorem.

  \begin{prop}\label{prop:properties_reducedFouriertransform}
    For all $v\in \mathcal S(\R^n)$ and $\xi\in\R^{n-1},\eta\in\R$ we have 
    \begin{align*}
      \cF_n^+(-\Delta v)(\xi,\eta)
      &= (|\xi|^2+\eta^2) (\cF_n^+ v) (\xi,\eta) + (2\pi)^{-\frac{1}{2}}\left(\i\eta \cF_{n-1}[v(\cdot,0)](\xi)  +
      \cF_{n-1}[v'(\cdot,0)](\xi) \right),  \\
      \cF_n^-(-\Delta v)(\xi,\eta)
      &= (|\xi|^2+\eta^2) (\cF_n^- v) (\xi,\eta) - (2\pi)^{-\frac{1}{2}}\left(\i\eta
      \cF_{n-1}[v(\cdot,0)](\xi) +\cF_{n-1}[v'(\cdot,0)](\xi) \right).
    \end{align*}
    Moreover, $\ran(\cF_n^+),\ran(\cF_n^-)$ are $L^2(\R^n)$-orthogonal to each
    other and $\cF_n = \cF_n^+ + \cF_n^-$. 
    Here we denote $\partial_y v(x, y) =: v'(x,y)$.
 %   (Here $'$ refers to the one-dimensional derivative with respect to the last variable.)
  \end{prop}
      
  We only comment on the orthogonality property. For $f,g\in \mathcal S(\R^n)$  Plancherel's identity
  implies 
  $$
    \skp{\cF_n^+ f}{\cF_n^- g}_{L^2(\R^n)}
    = \skp{\cF_n f_+}{\cF_n g_-}_{L^2(\R^n)}
    = \skp{f_+}{g_-}_{L^2(\R^n)}
    %= \int_{\R} F(y)\ov{G(y)}\, \mathrm{d}y
    = 0
  $$
  since the supports of $f_+,g_-$ intersect only in a null set.
  We introduce $\mu_j:= \sqrt{\lambda-V_j}>0$ and
  the complex-valued functions $\nu_{j,\eps}:\R^{n-1}\to \C$ via 
  $$
    \nu_{j,\eps}(\xi)^2
    = \mu_j^2-|\xi|^2+\i\eps
    = \lambda-V_j-|\xi|^2+\i\eps
    \quad\text{and}\quad \Imag(\nu_{j,\eps}(\xi))>0.
  $$
  Notice that $\nu_{j,\eps}(\xi)\to \nu_j(\xi)$ as $\eps\searrow 0$ where
 \begin{align}\label{eq:nu}
 \nu_j(\xi) :=
 \begin{cases}
  	 (\mu_j^2-|\xi|^2)^{\frac{1}{2}} & \quad\text{if } |\xi|\leq \mu_j,  \\
    \i(|\xi|^2-\mu_j^2)^{\frac{1}{2}} & \quad\text{if } |\xi|\geq \mu_j.
  \end{cases}
  \end{align}
   Later we will need the following elementary estimate:
  \begin{equation}\label{eq:estimate_nu}
    1+|\xi| \les |\nu_j(\xi)| \sqrt{1+|\nabla \nu_j(\xi)|^2} \les 1+|\xi| \qquad (\xi\in\R^{n-1}).
  \end{equation}

  \begin{prop}\label{prop:ResolventFormulanD}
    Let $\lambda>V_1>V_2$ and $f\in S(\R^n)$. Then, for any given $\eps>0$, the unique solution
    $u_\eps\in S(\R^n)$ of~\eqref{eq:complexnDNLH2} is given by
    \begin{align} \label{eq:formula_uesp_nD}
      \begin{aligned}
 	u_\eps(x,y)
   &= \cF_n^{-1} \left( \frac{\mathcal F_n^+f }{|\cdot|^2-\mu_1^2-\i\eps} + \frac{\mathcal
   F_n^-f }{|\cdot|^2-\mu_2^2 -\i\eps} \right)(x,y)  \\
     & \quad  + \cF_{n-1}^{-1} \left(    \e{\i |y|\nu_{1,\eps}}  (m_{1,\eps} g_{+,\eps} +  m_{2,\eps}
     g_{-,\eps})\right)(x) \\
     &\quad  + \cF_{n-1}^{-1} \left( \e{\i |y|\nu_{2,\eps}}  (m_{3,\eps} g_{+,\eps} +m_{4,\eps}
     g_{-,\eps})\right)(x)
    \end{aligned}
    \end{align}
    where $g_{+,\eps}(\xi)=\cF_n^+ f(\xi,-\nu_{1,\eps}(\xi))$,
    $g_{-,\eps}(\xi)=\cF_n^-f(\xi,\nu_{2,\eps}(\xi))$ and
    \begin{align} \label{eq:def_mjeps}
  \begin{aligned}
  m_{1,\eps}(\xi)
  &:= \frac{\i\sqrt{\pi/2}}{\nu_{1,\eps}(\xi)+\nu_{2,\eps}(\xi)}\cdot
  \left(\sign(y)-\frac{\nu_{2,\eps}(\xi)}{\nu_{1,\eps}(\xi)}\right),
  \\
  m_{2,\eps}(\xi) 
  &:= \frac{\i\sqrt{\pi/2}}{\nu_{1,\eps}(\xi)+\nu_{2,\eps}(\xi)}\cdot \left(1+\sign(y)\right),\\
  m_{3,\eps}(\xi)
  &:= \frac{\i\sqrt{\pi/2}}{\nu_{1,\eps}(\xi)+\nu_{2,\eps}(\xi)}\cdot \left(1-\sign(y)\right),\\
  m_{4,\eps}(\xi)
  &:= \frac{\i\sqrt{\pi/2}}{\nu_{1,\eps}(\xi)+\nu_{2,\eps}(\xi)}\cdot \left(-\sign(y)-\frac{\nu_{1,\eps}(\xi)}{\nu_{2,\eps}(\xi)}\right).
  \end{aligned}
\end{align}
%    \begin{align}\label{eq:formulaInitialvalues_nD}
%       \vecII{\alpha_{f,\eps}(\xi)}{\beta_{f,\eps}(\xi)}
%     =  \frac{\sqrt{2\pi}}{\nu_{1,\eps}(\xi)+\nu_{2,\eps}(\xi)}
%     \matII{\i }{\i }{\nu_{2,\eps}(\xi)}{-\nu_{1,\eps}(\xi)}
%     \vecII{\cF_n^+f(\xi,-\nu_{1,\eps}(\xi))}{\cF_n^-f(\xi,\nu_{2,\eps}(\xi))}.
%   \end{align}
  \end{prop}
  \begin{proof}
    From Proposition~\ref{prop:properties_reducedFouriertransform} we obtain  
    \begin{align*}
    (|\xi|^2+\eta^2 - (\mu_1^2+\i\eps)) \mathcal F_n^+u_\eps(\xi,\eta)+ (2\pi)^{-\frac{1}{2}} \left(\i\eta
    \cF_{n-1}[u_\eps(\cdot,0)](\xi)   + \cF_{n-1}[u_\eps'(\cdot,0)](\xi) \right)
    &= \mathcal F_n^+f(\xi,\eta),\\
    (|\xi|^2+\eta^2 - (\mu_2^2+\i\eps)) \mathcal F_n^- u_\eps(\xi,\eta)- (2\pi)^{-\frac{1}{2}} \left(\i\eta
    \cF_{n-1}[u_\eps(\cdot,0)](\xi)  + \cF_{n-1}[u_\eps'(\cdot,0)](\xi) \right)
    &= \mathcal F_n^-f(\xi,\eta).
    \end{align*}
    So $\nu_{j,\eps}(\xi)^2 = \mu_j^2+\i\eps-|\xi|^2$ yields the formula   
  \begin{align} \label{eq:nD_transformed_equations}
    \begin{aligned}
     \mathcal F_n^+ u_\eps(\xi,\eta)
    &= \frac{\mathcal F_n^+f(\xi,\eta)-(2\pi)^{-\frac{1}{2}} \left(  \i\eta
    \cF_{n-1}[u_\eps(\cdot,0)](\xi)  + \cF_{n-1}[u_\eps'(\cdot,0)](\xi) \right)}{\eta^2-\nu_{1,\eps}(\xi)^2},
    \\
    \mathcal F_n^- u_\eps(\xi,\eta)
    &= \frac{\mathcal F_n^-f (\xi,\eta)+ (2\pi)^{-\frac{1}{2}} \left(  \i\eta
    \cF_{n-1}[u_\eps(\cdot,0)](\xi)  + \cF_{n-1}[u_\eps'(\cdot,0)](\xi)  \right)}{\eta^2-\nu_{2,\eps}(\xi)^2}.
    \end{aligned}
  \end{align}
  We now exploit $\ran(\cF_n^+)\perp\ran(\cF_n^-)$ in order to compute
  $\cF_{n-1}[u_\eps(\cdot,0)](\xi) ,\cF_{n-1}[u_\eps'(\cdot,0)](\xi) $. The Residue Theorem
  gives for all $\phi\in\mathcal S(\R^{n-1})$ and $\zeta(z):=\sqrt{2\pi} \e{-|z|}$ for $z\in\R$
  \begin{align*}
    0
    &=  \skp{\mathcal F_n^- u_\eps}{\mathcal F_n^+(\phi\otimes \zeta)}_{L^2(\R^n)} \\
    &= \int_\R\int_{\R^{n-1}} \left(\frac{\mathcal F_n^-f(\xi,\eta)+  (2\pi)^{-\frac{1}{2}}
    (\i\eta \cF_{n-1}[u_\eps(\cdot,0)](\xi) +\cF_{n-1}[u_\eps'(\cdot,0)](\xi) )}{\eta^2-\nu_{2,\eps}(\xi)^2}
    \cdot \frac{\overline{\hat\phi(\xi)}}{-\i\eta+1} \right) \, \mathrm{d}\xi \, \mathrm{d}\eta \\
    &= \int_{\R^{n-1}}  \overline{\hat\phi(\xi)} \cdot \Bigg[ \left( \int_\R
    \frac{\mathcal F_n^-f(\xi,\eta)}{(\eta^2-\nu_{2,\eps}(\xi)^2)(-\i\eta+1)} \, \mathrm{d}\eta \right)   \\
    &\qquad\qquad\qquad + \frac{\cF_{n-1}[u_\eps'(\cdot,0)](\xi) }{\sqrt{2\pi}} \left( \int_\R
    \frac{1}{(\eta^2-\nu_{2,\eps}(\xi)^2)(-\i\eta+1)}\, \mathrm{d}\eta\right) \\
    &\qquad \qquad\qquad+ \frac{\cF_{n-1}[u_\eps(\cdot,0)](\xi) }{\sqrt{2\pi}}  \left( \int_\R
    \frac{\i \eta }{(\eta^2-\nu_{2,\eps}(\xi)^2)(-\i\eta+1)}\, \mathrm{d}\eta\right)  \Bigg]  \, \mathrm{d}\xi \\
    &= \int_{\R^{n-1}} \frac{\i \pi \overline{\hat\phi(\xi)}}{\nu_{2,\eps}(\xi)(1-\i\nu_{2,\eps}(\xi))} \cdot
    \Big( \mathcal F_n^-f(\xi,\nu_{2,\eps}(\xi))  +  \frac{
    \i \nu_{2,\eps}(\xi)\cF_{n-1}[u_\eps(\cdot,0)](\xi) +\cF_{n-1}[u_\eps'(\cdot,0)](\xi)}{\sqrt{2\pi}}  \Big) \,
    \mathrm{d}\xi \\
    \intertext{and similarly }
    0
    &= \skp{\mathcal F_n^+ u_\eps}{\mathcal F_n^-(\phi\otimes \zeta)}_{L^2(\R^n)} \\
    &= \int_{\R^{n-1}} \frac{\i \pi  \overline{\hat\phi(\xi)}}{\nu_{1,\eps}(\xi)(1-\i\nu_{1,\eps}(\xi))} \cdot
    \Big( \mathcal F_n^+ f(\xi,-\nu_{1,\eps}(\xi))  - \frac{
    \i \nu_{2,\eps}(\xi)\cF_{n-1}[u_\eps(\cdot,0)](\xi) +\cF_{n-1}[u_\eps'(\cdot,0)](\xi)}{\sqrt{2\pi}}    \Big) \, \mathrm{d}\xi.
  \end{align*}
  Since $\phi\in\mathcal S(\R^{n-1})$ was arbitrary, we get for almost all $\xi\in\R^{n-1}$
  \begin{align*} 
    \matII{-\i\nu_{1,\eps}(\xi)}{1}{-\i\nu_{2,\eps}(\xi)}{-1} \vecII{\cF_{n-1}[u_\eps(\cdot,0)](\xi) }{\cF_{n-1}[u_\eps'(\cdot,0)](\xi) }
    =\sqrt{2\pi} \vecII{\cF_n^+f(\xi,-\nu_{1,\eps}(\xi))}{\cF_n^-f(\xi,\nu_{2,\eps}(\xi))}
    =  \sqrt{2\pi} \vecII{g_{+,\eps}(\xi)}{g_{-,\eps}(\xi)}.
  \end{align*}
  Inverting this linear system we get
   \begin{align}\label{eq:formulaInitialvalues_nD}
     \vecII{\cF_{n-1}[u_\eps(\cdot,0)](\xi) }{\cF_{n-1}[u_\eps'(\cdot,0)](\xi) }
    =  \frac{\sqrt{2\pi}}{\nu_{1,\eps}(\xi)+\nu_{2,\eps}(\xi)}
    \matII{\i }{\i }{\nu_{2,\eps}(\xi)}{-\nu_{1,\eps}(\xi)}
    \vecII{g_{+,\eps}(\xi)}{g_{-,\eps}(\xi)}.
  \end{align}
  From this and $\cF_n^++\cF_n^- =\cF_n$  we get for $x\in\R^{n-1},y\in\R$
  \begin{align*}
    u_\eps(x,y)
    &= \cF_n^{-1}(\cF_n^+ u_\eps + \cF_n^-u_\eps)(x,y) \\
   &\stackrel{\eqref{eq:nD_transformed_equations}}= \cF_n^{-1} \left(
   \frac{\cF_n^+f(\xi,\eta)}{\eta^2-\nu_{1,\eps}(\xi)^2} + \frac{\cF_n^-f(\xi,\eta)}{\eta^2-\nu_{2,\eps}(\xi)^2} \right)(x,y)  \\
   &\quad -  \frac{1}{\sqrt{2\pi}}\cF_n^{-1} \left[\left( \frac{\i \eta}{\eta^2-\nu_{1,\eps}(\xi)^2} -
     \frac{\i \eta}{\eta^2-\nu_{2,\eps}(\xi)^2} \right) \cdot (\cF_n u_\eps)(\xi,0) 
     \right](x,y)   \\
   &\quad -  \frac{1}{\sqrt{2\pi}} \cF_n^{-1} \left[ \left( \frac{1}{\eta^2-\nu_{1,\eps}(\xi)^2} -
      \frac{1}{\eta^2-\nu_{2,\eps}(\xi)^2} \right) \cdot (\cF_nu_\eps)'(\xi,0)\right](x,y)  \\
   &= \cF_n^{-1} \left( \frac{\cF_n^+f}{|\cdot|^2- \mu_1^2-\i\eps}  +
   \frac{\cF_n^-f}{|\cdot|^2- \mu_2^2-\i\eps} \right)(x,y) 
    \\
   &\quad -  \frac{1}{2\pi} \cF_{n-1}^{-1}\left( \int_\R \left( \frac{\i \eta}{\eta^2-\nu_{1,\eps}(\xi)^2} -
     \frac{\i \eta}{\eta^2-\nu_{2,\eps}(\xi)^2} \right)\e{\i y\eta}\, \mathrm{d}\eta
     \cdot (\cF_n u_\eps)(\xi,0) \right)(x)\\
   &\quad  -  \frac{1}{2\pi} \cF_{n-1}^{-1}\left( \int_\R
   \left(\frac{1}{\eta^2-\nu_{1,\eps}(\xi)^2}-\frac{1}{\eta^2-\nu_{2,\eps}(\xi)^2}\right)\e{\i y\eta} \, \mathrm{d}\eta
    \cdot (\cF_nu_\eps)'(\xi,0)
    \right)(x) \\
   &= \cF_n^{-1} \left( \frac{\cF_n^+f}{|\cdot|^2-\mu_{1,\eps}^2} +
   \frac{\cF_n^-f}{|\cdot|^2-\mu_{2,\eps}^2} \right)(x,y) \\
   &\quad +  \frac{\sign(y)}{2} \cF_{n-1}^{-1}\left( \Big(\e{\i  |y|\nu_{1,\eps}}-\e{\i  |y|\nu_{2,\eps}} \Big)
   (\cF_n u_\eps)(\cdot,0)\right)(x)  \\
   &\quad  - \frac{\i }{2}  \cF_{n-1}^{-1}\left( \Big( \frac{\e{\i  |y|\nu_{1,\eps}}}{\nu_{1,\eps}} -
   \frac{\e{\i  |y|\nu_{2,\eps}}}{\nu_{2,\eps}} \Big) (\cF_nu_\eps)'(\cdot,0)\right)(x).
 \end{align*}
 Combining this identity with~\eqref{eq:formulaInitialvalues_nD} we
 find~\eqref{eq:formula_uesp_nD},\eqref{eq:def_mjeps}.
\end{proof}

 It will turn out useful to decompose the last lines of \eqref{eq:formula_uesp_nD} according to 
 \begin{align*}
    w_\eps(x,y) +\mathfrak{w}_\eps(x,y)+\mathfrak{W}_\eps(x,y)+W_\eps(x,y)
 \end{align*}
 where the small frequencies are collected in $w_{\eps}$, 
 the large ones in $W_{\eps}$ and the remaining intermediate ranges of frequencies are covered by the terms
 $\mathfrak{w}_{\eps}, \mathfrak{W}_{\eps}$. Formally,
  \begin{align*}% \label{eq:def_wjeps}
    %\begin{aligned}
    w_\eps(x,y)&:=  \cF_{n-1}^{-1}\left(  \e{\i|y|\nu_{1,\eps}} (\ind_{|\cdot|\leq \mu_1} 
     m_{1,\eps}  g_{+,\eps} +\ind_{|\cdot|\leq \mu_1}  m_{2,\eps} g_{-,\eps})\right)(x) \\
    &\qquad + \cF_{n-1}^{-1}\left( \e{\i|y|\nu_{2,\eps}} (\ind_{|\cdot|\leq \mu_1} m_{3,\eps}  g_{+,\eps}
    +  \ind_{|\cdot|\leq \mu_2} m_{4,\eps} g_{-,\eps})   \right)(x),  \\
    \mathfrak{w}_\eps(x,y)
    &:=  \cF_{n-1}^{-1}\left(
     \e{\i |y|\nu_{1,\eps}} \ind_{\mu_1<|\cdot|\leq \mu_2}  m_{2,\eps} g_{-,\eps} +
        \e{\i |y|\nu_{2,\eps}} \ind_{\mu_1<|\cdot|\leq \mu_2}  m_{3,\eps} g_{+,\eps}\right)(x), \\
    \mathfrak{W}_\eps(x,y)
    &:=\cF_{n-1}^{-1}\left(  \e{\i|y|\nu_{1,\eps}} \ind_{\mu_1<|\cdot|\leq \mu_1+\mu_2} 
    m_{1,\eps}  g_{+,\eps} 
    + \e{\i|y|\nu_{1,\eps}} \ind_{\mu_2<|\cdot|\leq \mu_1+\mu_2}  m_{2,\eps}  g_{-,\eps}\right)(x) \\
    &\qquad + \cF_{n-1}^{-1}\left( 
     \e{\i|y|\nu_{2,\eps}}\ind_{\mu_2<|\cdot|\leq \mu_1+\mu_2}    m_{3,\eps} g_{+,\eps}
    + \e{\i|y|\nu_{2,\eps}}\ind_{\mu_2<|\cdot|\leq \mu_1+\mu_2} m_{4,\eps} g_{-,\eps}   \right)(x),  \\
    W_\eps(x,y)&:= \cF_{n-1}^{-1}\left(
    \e{\i |y|\nu_{1,\eps}} (\ind_{|\cdot|> \mu_1 + \mu_2}
    m_{1,\eps} g_{+,\eps}+ \ind_{|\cdot|>\mu_1 + \mu_2}m_{2,\eps} g_{-,\eps})\right)(x) \\
    &\qquad + \cF_{n-1}^{-1}\left(\e{\i |y|\nu_{2,\eps}} (\ind_{|\cdot|> \mu_1 + \mu_2}
     m_{3,\eps} g_{+,\eps} + \ind_{|\cdot|> \mu_1 + \mu_2} m_{4,\eps} g_{-,\eps}) \right)(x).
    %\end{aligned}
  \end{align*}
  In the following section we will state estimates for $w_{\eps},\mathfrak w_{\eps},\mathfrak
  W_\eps, W_{\eps}$ (see the
  Propositions~\ref{prop:smallfreq},~\ref{prop:interferencefreq},~\ref{prop:intermediatefreq},~\ref{prop:largefreq})
  that lead to the proof of Theorem~\ref{thm:1}. Before going on with this we compute the limit of
  $u_\eps$ as $\eps\searrow 0$.  The above representation formula for $u_\eps$ leads to the definition
  \begin{align} \label{eq:formula_u+}
    \begin{aligned}
    \big(\mathcal R(\lambda+\i 0)f\big)(x,y)
    &:= u_+(x,y) \\
    &:= \lim_{\eps\searrow 0} \cF_n^{-1} \left( \frac{\cF_n^+ f}{|\,\cdot\,|^2 -\mu_1^2-\i\eps} +
   \frac{\cF_n^-f}{|\,\cdot\,|^2-\mu_2^2-\i\eps} \right)(x,y) \\
   &\quad \:  
     + \cF_{n-1}^{-1} \left(
    \e{\i |y|\nu_1}  (m_1 g_+ +  m_2g_-)\right)(x) \\
   &\quad \:     
     + \cF_{n-1}^{-1} \left(
      \e{\i |y|\nu_2}  (m_3 g_++m_4g_-)\right)(x)
   \end{aligned}
  \end{align}
  where the limit in the first line will be a weak limit in $L^q(\R^n)$.
%   Notice that
%   \begin{align}\label{eq:formulaInitialvalues_nDeps=0}
%       \vecII{\alpha_{f}(\xi)}{\beta_{f}(\xi)}
%      =  \frac{\sqrt{2\pi}}{\nu_{1}(\xi)+\nu_{2}(\xi)}
%      \matII{\i }{\i }{\nu_{2}(\xi)}{-\nu_{1}(\xi)}
%      \vecII{\cF_n^+f(\xi,-\nu_{1}(\xi))}{\cF_n^-f(\xi,\nu_{2}(\xi))}.
%   \end{align}
  As above, the last two lines of \eqref{eq:formula_u+} can be rewritten as
  \begin{align*}
     w(x,y) +\mathfrak{w}(x,y)+\mathfrak{W}(x,y)+W(x,y)
  \end{align*}
  where  $g_+(\xi)=\cF_n^+ f(\xi,-\nu_1(\xi))$, $g_-(\xi)=\cF_n^-f(\xi,\nu_2(\xi))$ and
  \begin{align} \label{def_wjWj}
    \begin{aligned}
    w(x,y)&:=  \cF_{n-1}^{-1}\left(  \e{\i|y|\nu_{1}} (\ind_{|\cdot|\leq \mu_1} 
     m_1  g_+ +\ind_{|\cdot|\leq \mu_1}  m_2 g_-)\right)(x) \\
    &\qquad + \cF_{n-1}^{-1}\left( \e{\i|y|\nu_{2}} (\ind_{|\cdot|\leq \mu_1} m_3  g_+
    +  \ind_{|\cdot|\leq \mu_2} m_4 g_-)   \right)(x),  \\
    \mathfrak{w}(x,y)
    &:=  \cF_{n-1}^{-1}\left(
     \e{\i |y|\nu_{1}} \ind_{\mu_1<|\cdot|\leq \mu_2}  m_2 g_- +
        \e{\i |y|\nu_{2}} \ind_{\mu_1<|\cdot|\leq \mu_2}  m_3 g_+\right)(x), \\
    \mathfrak{W}(x,y)
    &:=\cF_{n-1}^{-1}\left(  \e{\i|y|\nu_{1}} \ind_{\mu_1<|\cdot|\leq \mu_1+\mu_2} 
    m_1  g_+ 
    + \e{\i|y|\nu_{1}} \ind_{\mu_2<|\cdot|\leq \mu_1+\mu_2}  m_2  g_-\right)(x) \\
    &\qquad + \cF_{n-1}^{-1}\left( 
     \e{\i|y|\nu_{2}}\ind_{\mu_2<|\cdot|\leq \mu_1+\mu_2}    m_3 g_+
    + \e{\i|y|\nu_{2}}\ind_{\mu_2<|\cdot|\leq \mu_1+\mu_2} m_4 g_-   \right)(x),  \\
    W(x,y)&:= \cF_{n-1}^{-1}\left(
    \e{\i |y|\nu_{1}} (\ind_{|\cdot|> \mu_1 + \mu_2}
    m_1 g_++ \ind_{|\cdot|>\mu_1 + \mu_2}m_2 g_-)\right)(x) \\
    &\qquad + \cF_{n-1}^{-1}\left(\e{\i |y|\nu_{2}} (\ind_{|\cdot|> \mu_1 + \mu_2}
     m_3 g_+ + \ind_{|\cdot|> \mu_1 + \mu_2} m_4 g_-) \right)(x).
    \end{aligned}
  \end{align}
  Here,
   \begin{align} \label{eq:def_mj}
  \begin{aligned}
  m_1(\xi)
  &:= \frac{\i\sqrt{\pi/2}}{\nu_1(\xi)+\nu_2(\xi)}\cdot \left(\sign(y)-\frac{\nu_2(\xi)}{\nu_1(\xi)}\right),
  \\
  m_2(\xi) 
  &:= \frac{\i\sqrt{\pi/2}}{\nu_1(\xi)+\nu_2(\xi)}\cdot \left(1+\sign(y)\right),\\
  m_3(\xi)
  &:= \frac{\i\sqrt{\pi/2}}{\nu_1(\xi)+\nu_2(\xi)}\cdot \left(1-\sign(y)\right),\\
  m_4(\xi)
  &:= \frac{\i\sqrt{\pi/2}}{\nu_1(\xi)+\nu_2(\xi)}\cdot \left(-\sign(y)-\frac{\nu_1(\xi)}{\nu_2(\xi)}\right).
  \end{aligned}
\end{align}
Notice that in the case without a jump in the potential we have $\mu_1=\mu_2=:\mu,\nu_1\equiv \nu_2$ and
the formula \eqref{eq:formula_u+} simplifies to 
$$
  \big(\mathcal R(\lambda+\i 0)f\big)(x,y)
  = \lim_{\eps\searrow 0} \cF_n^{-1} \left( \frac{\cF_n f}{|\,\cdot\,|^2 -\mu^2-\i\eps}\right)(x,y)  
      \qquad\text{if }V_1=V_2 
$$
because of $m_1+m_3\equiv m_2+m_4\equiv 0$.
   At several places we shall need the estimates 
  \begin{align}\label{eq:mj_estimates}
    \begin{aligned}
    |m_1(\xi)|&\les (1+|\xi|)^{-1}|\nu_1(\xi)|^{-1} &&\qquad (\xi\in\R^{n-1}),\\ 
    |m_2(\xi)|+|m_3(\xi)|&\les (1+|\xi|)^{-1} &&\qquad (\xi\in\R^{n-1}), \\
    |m_4(\xi)| &\les (1+|\xi|)^{-1}|\nu_2(\xi)|^{-1}
    &&\qquad (\xi\in\R^{n-1}).
  \end{aligned}
  \end{align}

\section{Proof of Theorem~\ref{thm:1} and Corollary~\ref{cor:2}}  \label{sec:ProofThm}
 
  We first collect a few tools from Harmonic Analysis that we will need in our estimates. For $n \geq 3$, the
  first line in~\eqref{eq:formula_uesp_nD} and~\eqref{eq:formula_u+} may be analyzed with the aid of
  Guti\'{e}rrez' Limiting Absorption Principle~\cite[Theorem~6]{Gut_Nontrivial} 
  (see also \cite[Theorem~2.3]{KenRuSo_Uniform})
  for the Helmholtz equation with constant coefficients in $\R^n$. The corresponding result for the case $n=2$ was provided
  by~Ev\'{e}quoz~\cite[Theorem~2.1]{Eveq_plane}.  
  
  \begin{thm}[Guti\'{e}rrez, Ev\'{e}quoz] \label{thm:Gutierrez} 
    Assume $n\in\N, n\geq 2,(p,q)\in\mathcal D$ and $V\equiv V_1=V_2$. If $\lambda>V_1=V_2$ then the solutions
    $u_\eps$ of~\eqref{eq:complexnDNLH2} and $u_+, u_-:=\ov{u_+}$ from~\eqref{eq:formula_u+} satisfy 
    $$
      \|u_+\|_{L^q(\R^n)} + \|u_-\|_{L^q(\R^n)} 
      + \sup_{0<|\eps|\leq 1} \|u_\eps\|_{L^q(\R^n)} \les \|f\|_{L^p(\R^n)}.
    $$
    %The same is true in case $n=2$ if additionally $\frac{1}{p}-\frac{1}{q}<\frac{2}{n}$ is assumed.
  \end{thm}
  
  More can be said about the qualitative properties of $u_+,u_-$, especially concerning their behaviour at
  infinity which is governed by an outgoing respectively incoming Sommerfeld radiation condition that even
  characterize these solutions of the Helmholtz equation, see for instance \cite[Corollary~1]{Gut_Nontrivial}
  in the case $n\geq 3$. As solutions of the Helmholtz equation~\eqref{eq:nDNLH}, the imaginary parts of
  $u_\pm$ are solutions of the homogeneous Helmholtz equation. Computations reveal (see for instance (5.6)
  in~\cite{Ruiz_LN}) that $\Imag(u_+)=-\Imag(u_-)$ is a multiple of the function
  $\cF_d^{-1}(\cF_d f\:\mathrm{d}\sigma_\mu)$ where $\mu=\sqrt{\lambda-V_1}=\sqrt{\lambda-V_2}>0$. This
  corresponds to a Herglotz wave given by the density $\cF_d f$ on $\Ss_\mu^{d-1}$. So Theorem~\ref{thm:Gutierrez} implies
  the following.
  
  \begin{cor} \label{cor:Gutierrez}
    For $n\in\N,n\geq 2$ and $(p,q)\in\mathcal D$  the linear operator
    $f\mapsto \cF_n^{-1}(\cF_n f\:\mathrm{d}\sigma_\mu)$ is  bounded from $L^p(\R^d)$ to $L^q(\R^d)$ for all $\mu>0$. 
  \end{cor}
  
  Another reference for this result and for the optimality of the asserted range can be found in \cite[Theorem 2.14]{KwonLee_Sharp}. 
  We will also need several Fourier restriction theorems for the Fourier
  transforms $\cF_n,\cF_{n-1}$ restricted to spheres in $\R^{n-1}$ respectively $\R^n$. We use $d$ as the
  dimensional parameter.  
   
  \begin{thm}[Stein-Tomas] \label{thm:SteinTomas}
    Let $d\in\N, d\geq 2$ and $1\leq p\leq \frac{2(d+1)}{d+3}$, $\mu>0$. Then %the following estimates
    % hold
    $$
      \|\cF_d f\|_{L^2(\Ss^{d-1}_\mu)} \les \mu^{\frac{d-1}{2}-\frac{d}{p'}}  \|f\|_{L^p(\R^d)},\qquad
      \|\cF_d(g\:\mathrm{d}\sigma_\mu)\|_{L^{p'}(\R^d)} \les \mu^{\frac{d-1}{2}-\frac{d}{p'}} \|g\|_{L^2(\Ss^{d-1}_\mu)}.
    $$
  \end{thm}
  
  The Stein-Tomas Theorem (see \cite{Tomas_Restriction} or \cite[p.386]{Stein_Harmonic}) is one particularly
  important estimate that embeds into a whole family of estimates. The Restriction Conjecture says that the
  estimates
  \begin{equation}\label{eq:RC}
       \|\cF_d h\|_{L^{q'}(\Ss^{d-1}_\mu)} \les \mu^{\frac{d-1}{q'}-\frac{d}{p'}} \|h\|_{L^p(\R^d)},\qquad
      \|\mathcal F_d(g\:\mathrm{d}\sigma_\mu)\|_{L^{p'}(\R^d)} \les
      \mu^{\frac{d-1}{q'}-\frac{d}{p'}} \|g\|_{L^q(\Ss^{d-1}_\mu)}.
  \end{equation}
  hold whenever $p'>\frac{2d}{d-1}, q \geq \left(\frac{d-1}{d+1} \: p' \right)'$. In the two-dimensional case
  $d=2$ the validity of~\eqref{eq:RC} is known since the 1970s
  \cite[Theorem~3]{Zyg_OnFourier},\cite[p.33-34]{Feff_Inequalities}. In the higher-dimensional case, however,
  the conjecture is still unsolved. Up to our knowledge, the strongest known result in this direction is due to Tao, see~\cite[Figure~3]{Tao} and \cite[p.1382]{Tao_Bilinear}.  
  
  \begin{thm}[Fefferman, Zygmund] \label{thm:RC}
    Let $d=2$ and  $p'>\frac{2d}{d-1},q \geq \left(\frac{d-1}{d+1} \: p' \right)',
    \mu>0$.  Then~\eqref{eq:RC} holds.   
  \end{thm}
  
   \begin{thm}[Tao] \label{thm:Tao}
    Let $d\in\N,d\geq 3$ and $p'> \frac{2(d+2)}{d}, q \geq \left(\frac{d-1}{d+1} \: p' \right)',
    \mu>0$.  Then~\eqref{eq:RC} holds. 
  \end{thm}
  
  As a consequence, in our analysis we can use \eqref{eq:RC} for $p'>p_*(d), q \geq
  \left(\frac{d-1}{d+1} \: p' \right)'$. We recall $p_*(d)=q_*(d)'=\frac{2d}{d+1}$ in the case $d=2$
  and $p_*(d)=q_*(d)'=\frac{2(d+2)}{d+4}$ in the case $d\geq 3$.
  All other major technical results are contained in Theorem~\ref{thm:Tlambdaalpha_d=1} and
  Theorem~\ref{thm:Tlambdaalpha} that we presented in the Introduction. In view of the
  representation formula~\eqref{eq:formula_u+} and the following remarks we will demonstrate
  Theorem~\ref{thm:1} by a separate discussion for four different frequency regimes. 
  Our result for the small frequency parts are the
  following.

  \begin{prop} \label{prop:smallfreq}
    Let $n\in\N, n\geq 2$ and $\frac{1}{p}>\frac{1}{p_*(n)}$, $\frac{1}{q}<\frac{1}{q_*(n)}$,
    $\frac{1}{p}-\frac{1}{q}\geq \frac{2}{n+1}$. Then we have  
    $$
      \|w\|_{L^q(\R^n)}  
     + \sup_{0<|\eps|\leq 1}  \|w_{\eps}\|_{L^q(\R^n)} \les \|f\|_{L^p(\R^n)}.
    $$
    In particular, this holds for all $(p,q)\in \tilde{\mathcal D}$.
    If the Restriction Conjecture is true, then these estimates even hold whenever 
    $\frac{1}{p}>\frac{n+1}{2n},\frac{1}{q}<\frac{n-1}{2n},\frac{1}{p}-\frac{1}{q}\geq
    \frac{2}{n+1}$ and hence for all $(p,q)\in\mathcal D$.    
  \end{prop}
  
  The proof of Proposition~\ref{prop:smallfreq} will be given in Section~\ref{sec:smallfreq}.  
  In Section~\ref{sec:interferencefreq}, we analyze the first of the two terms containing intermediate frequencies. We will prove the following result.
  
  \begin{prop} \label{prop:interferencefreq}
    Let $n\in \N,n\geq 2$ and  $\frac{1}{p}>\frac{1}{p_*(n)}$, $\frac{1}{q}<\frac{1}{q_*(n)}$,
    $\frac{1}{p}-\frac{1}{q}\geq \frac{2}{n+1}$. Then we have 
    $$
      \|\mathfrak{w} \|_{L^q(\R^n)}  
     + \sup_{0<|\eps|\leq 1}  \|\mathfrak{w}_{\eps}\|_{L^q(\R^n)} \les \|f\|_{L^p(\R^n)}.
    $$
    In particular, this holds for all $(p,q)\in \tilde{\mathcal D}$.
    If the Restriction Conjecture is true, then this estimate even holds for  
    $\frac{1}{p}>\frac{n+1}{2n},\frac{1}{q}<\frac{n-1}{2n},\frac{1}{p}-\frac{1}{q}\geq
    \frac{2}{n+1}$ and thus for all $(p,q)\in\mathcal D$.
  \end{prop}  
  
  The estimates related to the second range of intermediate frequencies rely
  on Theorem~\ref{thm:Tlambdaalpha_d=1} $(n=2)$ and Theorem~\ref{thm:Tlambdaalpha} ($n\geq 3$). The proof is
  provided in Section~\ref{sec:intermediatefreq}.

  \begin{prop} \label{prop:intermediatefreq}
    Let $n\in\N, n\geq 2$ and $\frac{1}{p}>\frac{n+1}{2n}$,
    $\frac{1}{q}<\frac{n-1}{2n}$, $\frac{1}{p}-\frac{1}{q}\geq \frac{2}{n+1}$. Then we have
    $$
      \|\mathfrak{W}\|_{L^q(\R^n)}   
     + \sup_{0<|\eps|\leq 1}  \|\mathfrak{W}_{\eps}\|_{L^q(\R^n)}  \les \|f\|_{L^p(\R^n)}.
    $$
    In particular, this holds for all 
   % \r{$(p,q)\in \tilde{\mathcal D}$  and for all} 
    $(p,q)\in\mathcal D$.   
  \end{prop}

  The estimates for the large frequency parts $W,W_{\eps}$ are the easiest ones.
  The proof will be presented in Section~\ref{sec:largefreq}.  
 
  \begin{prop} \label{prop:largefreq}
	Let $n\in\N, n\geq 2$ and $1\leq p\leq 2\leq q\leq \infty$ with 
	$0 \leq \frac{1}{p} - \frac{1}{q} \leq \frac{2}{n}$ and 
	$\frac{1}{p} - \frac{1}{q}<\frac{2}{n}$ if $p=1$ or $q=\infty$.  Then we have 
	$$
  	  \|W\|_{L^q(\R^n)}+ \sup_{0<|\eps|\leq 1}  \norm{W_{\eps}}_{L^q(\R^n)}  
      \les \|f\|_{L^p(\R^n)}.
	$$
	In particular, this holds for all 
	%\r{$(p,q)\in \tilde{\mathcal D}$ and for all} 
	$(p,q)\in\mathcal D$.  
  \end{prop}

 \medskip

\textbf{Proof of Theorem~\ref{thm:1}.} 
From  Proposition~\ref{prop:ResolventFormulanD} and the representation
formulas~\eqref{eq:formula_uesp_nD},~\eqref{eq:formula_u+} we get 
\begin{align*}% \label{eq:proof-thm1}
%\begin{aligned}
  &\|u_+\|_{L^q(\R^n)} + \|u_-\|_{L^q(\R^n)}
  + \sup_{0<|\eps|\leq 1} \|u_\eps\|_{L^q(\R^n)}\ \\ 
  &\les \sup_{0<|\eps|\leq 1}  \norm{\cF_n^{-1} \left( \frac{\cF_n^+ f}{|\cdot|^2 -\mu_1^2-\i\eps} + 
   \frac{\cF_n^-f}{|\cdot|^2-\mu_2^2-\i\eps} \right)}_{L^q(\R^n)} \\
  &\quad +  \norm{w}_{L^q(\R^n)} + \norm{\mathfrak w}_{L^q(\R^n)} + \norm{\mathfrak W}_{L^q(\R^n)} 
  + \norm{W}_{L^q(\R^n)}  \\
  &\quad +    \sup_{0<|\eps|\leq 1} \left( \norm{w_\eps}_{L^q(\R^n)} + \norm{\mathfrak w_\eps}_{L^q(\R^n)} +
  \norm{\mathfrak W_\eps}_{L^q(\R^n)} + \norm{W_\eps}_{L^q(\R^n)}  \right).
%\end{aligned}
\end{align*}
For all exponents $(p,q)\in \tilde{\mathcal D}$  the control of each of these
terms by the $L^p$-norm of the right hand side is a consequence of Theorem~\ref{thm:Gutierrez},
Proposition~\ref{prop:smallfreq}, Proposition~\ref{prop:interferencefreq},
Proposition~\ref{prop:intermediatefreq} and Proposition~\ref{prop:largefreq} because of $\tilde{\mathcal
D}\subset\mathcal D$.
If the Restriction Conjecture is true, then the same statement holds even all $(p,q)\in \mathcal D$. \qed

\bigskip       

\textbf{Proof of Corollary~\ref{cor:2}.} We briefly recall the dual variational technique for nonlinear
Helmholtz equations from~\cite{EvWe}. We aim at proving the existence of a real-valued function $u \in L^q(\R^n)$ satisfying 
\begin{align}\label{eq:nonlinear}
 -\Delta u +Vu - \lambda u =  \Gamma |u|^{q-2}u \qquad\text{in }\R^n 
\end{align}
in the distributional sense. In view of elliptic regularity theory any
distributional solution of such an equation will actually belong to $W^{2,r}_{\loc}(\R^n)$ for all $r\in
[1,\infty)$. Such solutions of the nonlinear PDE~\eqref{eq:nonlinear} will be obtained by solving the integral
equation $u  = K(\Gamma|u|^{q-2}u)$ where $K\phi:= \Real(\mathcal R(\lambda+\i 0)\phi)$ and $\mathcal R(\lambda+\i 0)$ has the mapping properties stated in
Theorem~\ref{thm:1}. We set $v:= \Gamma^{\frac{1}{q'}} |u|^{q-2}u$ and thus look for
 $v\in L^{q'}(\R^n)$ satisfying
$$
    |v|^{q'-2}v =  \Gamma^{\frac{1}{q}} K(\Gamma^{\frac{1}{q}} v).
$$
Since $K$ is symmetric, this equation has a variational structure. 
So we have to prove the existence of a nontrivial critical point of the functional 
$$
  I(v) := \frac{1}{q'} \int_{\R^n} |v|^{q'} - \frac{1}{2} \int_{\R^n} \big(\Gamma^{\frac{1}{q}}
  v\big) \left[ K\big(\Gamma^{\frac{1}{q}} v\big)\right]. 
$$ 
This functional has the Mountain Pass geometry, as we will explain and verify below. 
Moreover, exploiting $\Gamma\to 0$ at infinity, it satisfies the Palais-Smale
condition. This can be shown exactly as in \cite[Lemma~5.2]{EvWe} where the corresponding statement is proved in the special case $V_1=V_2$.
With these two ingredients we may apply the Mountain Pass Theorem \cite[Theorem~2.1]{AmbRab_dual} and obtain a nontrivial critical point $v$ of $I$. Transforming this function back according to $v= \Gamma^{\frac{1}{q'}} |u|^{q-2}u$, we get a nontrivial solution $u = \Gamma^{-\frac{1}{q}} |v|^{q'-2} v = K (\Gamma^{\frac{1}{q}} v) \in L^q(\R^n)$ of the nonlinear Helmholtz equation~\eqref{eq:NLH}.

\medskip

We now check that $I$ has the Mountain Pass geometry. First, by
 choice of $q$ in Corollary~\ref{cor:2}, the operator $\mathcal R(\lambda+\i 0):L^{q'}(\R^n)\to L^q(\R^n)$ is
bounded and thus  $K:L^{q'}(\R^n)\to L^q(\R^n)$ is bounded as well. Moreover, 
$$
  I(v) 
  \geq \frac{1}{q} \|v\|_{L^{q'}(\R^n)}^{q'} - \frac{1}{2} \|K\|_{L^{q'}(\R^n)\to L^q(\R^n)}
  \|\Gamma\|_{L^\infty(\R^n)}^{\frac{2}{q}} \|v\|_{L^{q'}(\R^n)}^2
$$
and $q'<2$ imply $I(0) = 0<\inf_{S_\varrho} I$ for some
sufficiently small $\varrho > 0$ where $S_\varrho$ denotes the sphere in $L^{q'}(\R^n)$ with radius $\varrho$. 
Finally, $I(tv)\to -\infty$ as $t\to\infty$ for some $v\in L^{q'}(\R^n)$, the proof of which will take the
remainder of this section. We adapt an idea from \cite[Section~3]{MaMoPe_Oscillating} and choose the ansatz $v=v_\delta$ where
\begin{equation}\label{eq:def_vdelta}
  v_\delta (x,y):= \Gamma(x,y)^{-\frac{1}{q}}w(x)\e{-y} 1_{(\delta,\infty)}(\Gamma(x,y)) 1_{(0, \infty)}(y)
  \qquad
  (x\in\R^{n-1},y\in\R,\delta> 0)
\end{equation} 
with sufficiently small $\delta > 0$ and with a nontrivial Schwartz function $w$ satisfying 
$\supp(\hat w) \subset \R^n \sm \ov{B_{\mu_2}(0)} = \{ \xi \in \R^{n-1}: |\xi|> \mu_2 \}$. Notice that
$v_\delta\in L^{q'}(\R^n)$ because of $\delta>0$ and 
$$
  \Gamma^{\frac{1}{q}}v_\delta \to f\quad\text{in }L^{q'}(\R^n) \quad \text{as } \delta \searrow 0
  \qquad\text{where }f(x,y)=w(x)\e{-y}1_{(0, \infty)}(y). 
$$  
Here we used $\Gamma>0$ on $\R^n$. So we find with the aid of Plancherel's theorem
\begin{align*}
  &\lim_{\delta \searrow 0} \int_{\R^n} \big(\Gamma^{\frac{1}{q}}
  v_\delta \big) \left[ K\big(\Gamma^{\frac{1}{q}} v_\delta \big)\right] \: \mathrm{d}(x,y) \\
  &= \int_{\R^n} f (Kf) \: \mathrm{d}(x,y) \\
  &= \Real\left(\int_{\R^n} \left(\mathcal R(\lambda+\i 0) f\right)\cdot f  \: \mathrm{d}(x,y)\right)\\
  &\stackrel{\eqref{eq:formula_u+}}= \Real\left(\int_{\R^n} \cF_n^{-1}\left(  
  \frac{\cF_n^+f  }{|\,\cdot\,|^2 -\mu_1^2-\i 0}  
  + \frac{\cF_n^+f  }{|\,\cdot\,|^2 -\mu_2^2-\i 0}
  \right)(x,y)\cdot f(x,y)\: \mathrm{d}(x,y) \right)\\
  &\quad +   \Real\left(\int_{\R^n}  \cF_{n-1}^{-1}\left(
   \e{\i|y|\nu_1}(m_1g_++m_2g_-)
   \right)(x) \cdot f(x,y) \: \mathrm{d}(x,y) \right) \\
  &\quad +   \Real\left(\int_{\R^n}  \cF_{n-1}^{-1}\left(
  \e{\i|y|\nu_1}(m_3g_++m_4g_-)
   \right)(x) \cdot f(x,y)
  \: \mathrm{d}(x,y) \right)\\
  &= \Real\left( \int_{\R^n} \frac{
  \cF_n^+f(\xi,\eta) \cdot \ov{\cF_n f(\xi,\eta)}}{|\xi|^2+\eta^2 -\mu_1^2-\i 0} + \frac{
  \cF_n^-f(\xi,\eta)\cdot \ov{\mathcal F_n f(\xi,\eta)}}{|\xi|^2+\eta^2-\mu_2^2-\i 0} \: \mathrm{d}(\xi,\eta) \right) 
  \\
  &\quad +  \Real\left(  \int_{\R^n}    
  \e{\i|y|\nu_1(\xi)}(m_1(\xi)g_+(\xi)+m_2(\xi)g_-(\xi)) \ov{\mathcal F_{n-1}[f(\cdot,y)](\xi)} \:
  \mathrm{d}(\xi,y) \right)  \\
  &\quad + \Real\left(     \int_{\R^n}  \e{\i|y|\nu_2(\xi)}(m_3(\xi)g_+(\xi)+m_4(\xi)g_-(\xi))
  \ov{\mathcal F_{n-1}[f(\cdot,y)](\xi)} \: \mathrm{d}(\xi,y) \right).
\end{align*}  
Inserting~\eqref{eq:def_vdelta} we get 
$$
   \cF_{n-1}[f(\cdot,y)](\xi)= \hat w(\xi)\e{-y}1_{(0, \infty)}(y),\quad
  \cF_n^+ f(\xi,\eta)=\cF_n f(\xi,\eta)= \frac{\hat w(\xi)}{1+\i\eta}, \quad \cF_n^-f\equiv 0.
$$ 
So our choice of $w$ implies 
$|\xi|^2+\eta^2\geq |\xi|^2>\mu_2^2>\mu_1^2$ for all $(\xi,\eta)\in \supp(\hat w)\times\R=\supp(\cF_n^+ f) =
\supp(\cF_n f)$. This has the following consequences:
\begin{itemize}
  \item[(i)] The principal value symbol $-\i 0$ can be omitted in the first two integrals.
  \item[(ii)] $\nu_j(\xi)=\i|\nu_j(\xi)|$ for $j=1,2$, see~\eqref{eq:nu}.
  \item[(iii)] $g_-(\xi)=\cF_n^-f(\xi,\nu_2(\xi))=0$ and $g_+(\xi)= \cF_n^+f(\xi,-\nu_1(\xi))=\frac{\hat
  w(\xi)}{1-\i\nu_1(\xi)}=\frac{\hat w(\xi)}{1+|\nu_1(\xi)|}$.
\end{itemize} 
Given that  $m_1(\xi)$ is real-valued and positive and
$m_3\equiv 0$  for $|\xi|\geq \mu_2>\mu_1,y>0$, see~\eqref{eq:def_mj} and (ii), this implies
\begin{align*}
  &\int_{\R^n} f (Kf) \: \mathrm{d}(x,y) \\ 
  &=  \int_{\R^n} \frac{
  |\cF_n f(\xi,\eta)|^2}{|\xi|^2+\eta^2 -\mu_1^2}\: \mathrm{d}(\xi,\eta)   
  +  \Real\left(  \int_{\R^n}    
  \e{\i|y|\nu_1(\xi)} m_1(\xi)g_+(\xi) \ov{\mathcal F_{n-1}[f(\cdot,y)](\xi)} \:
  \mathrm{d}(\xi,y) \right) \\
  &= \int_{\R^n} \frac{
  |\cF_n f(\xi,\eta)|^2}{|\xi|^2+\eta^2 -\mu_1^2}\: \mathrm{d}(\xi,\eta)   
  +  \int_{\R^{n-1}}  \frac{m_1(\xi) |\hat w(\xi)|^2}{1+|\nu_1(\xi)|}
   \left( \int_0^\infty \e{-(|\nu_1(\xi)|+1)y} \: \mathrm{d}y \right)  
    \: \mathrm{d}\xi    \\
   &= \int_{\R^n} \frac{
  |\cF_n f(\xi,\eta)|^2}{|\xi|^2+\eta^2 -\mu_1^2}\: \mathrm{d}(\xi,\eta)   +
  \int_{\R^{n-1}} \frac{m_1(\xi)|\hat w(\xi)|^2}{(1+|\nu_1(\xi)|)^2}\dxi >0.
\end{align*}  
 As a consequence, we obtain
$$
  \int_{\R^n} \big(\Gamma^{\frac{1}{q}} v_\delta \big) \left[ K\big(\Gamma^{\frac{1}{q}} v_\delta
  \big)\right] \: \mathrm{d}(x,y) >0 
$$
provided that $\delta>0$ is small enough. This finishes the proof of the Mountain Pass Geometry and the claim
is proved.
\qed

\section{Proof of Theorem~\ref{thm:Tlambdaalpha_d=1}} \label{sec:ProofTlambda_d=1}

  In this section we discuss the mapping properties of the operator $T_{\lambda,\alpha}$
  from~\eqref{eq:def_Tlambdaalpha} in the one-dimensional case $d=1$. Before proving the claims from
  Theorem~\ref{thm:Tlambdaalpha_d=1} on that matter, we provide two auxiliary results dealing with singular
  one-dimensional oscillatory integrals. We use the following well-known estimate.
  
  \begin{prop}[VIII.{$\S$}1.1.2 Corollary~\cite{Stein_Harmonic} on p.334]  \label{prop:Stein}
    Let $I\subset\R$ be an interval. Then we have for all $b\in W^{1,1}(I)$ the estimate 
    $$
      \left| \int_I \e{\i c \rho} b(\rho)\drho \right| 
      \les c^{-1} (|b(0)|+\|b'\|_{L^1(I)})
    $$
    with a constant independent of $I$ and $b$.
  \end{prop}

  \begin{prop} \label{prop:asymptotics_delta>0}
    Let $\delta\in (0,1)$ and $a\in C^1([0,1])$. Then the following holds for $c>0$: 
      \begin{align*}
       &\text{(i)}\qquad \left|\int_0^1  \e{\i c\rho} a(\rho) \rho^{-\delta} \:\mathrm{d}\rho\right|   
       \les  (1+c)^{\delta-1} (|a(0)|+\|a'\|_\infty).   \\ 
     &\text{(ii)}\qquad \left|\int_0^1  \e{\i c\rho} a(\rho) \rho^{-\delta} \:\mathrm{d}\rho
      -  a(0) c^{\delta-1}\int_0^\infty \e{\i \rho}\rho^{-\delta}\:\mathrm{d}\rho\right|  
     \les  c^{-1}\left(|a(0)|+\|a'\|_\infty \right).
   \intertext{In particular, there is $M>0$ independent of $a$ such that}     
    &\qquad\qquad  \left| \int_0^1 \e{\i c\rho} a(\rho)\rho^{-\delta}\drho\right|
     \gtrsim c^{\delta-1}|a(0)|
     \qquad\text{for }c \geq M(1+\|a'\|_\infty |a(0)|^{-1})^{1/\delta}.
     \end{align*}
  \end{prop}  
  \begin{proof}
    For $0<c\leq 1$ the estimate (i) is trivial. For $c>0$ we get from Proposition~\ref{prop:Stein} 
    \begin{align*}
    \left|\int_0^1  \e{\i c\rho} a(\rho) \rho^{-\delta} \:\mathrm{d}\rho \right| 
    &\les  \left|\int_0^1  \e{\i  c\rho}  (a(\rho)-a(0)) \rho^{-\delta} \drho\right|
      +   |a(0)| \left| \int_0^1  \e{\i  c\rho} \rho^{-\delta} \drho \right|   \\
    &\les  c^{-1} \int_0^1  |((a(\rho)-a(0))\rho^{-\delta})'| \drho
      +    |a(0)| c^{\delta-1} \left| \int_0^c  \e{\i \rho} \rho^{-\delta} \drho \right|  \\
    &\leq  c^{\delta-1} (|a(0)|+\|a'\|_\infty). 
   \end{align*}
  The estimate (ii) is similar. For $0<c\leq 1$ the estimate is trivial, while for $c>0$ we may exploit
  Proposition~\ref{prop:Stein} once more to get
  \begin{align*}
    &\left|\int_0^1  \e{\i  c\rho}  a(\rho) \rho^{-\delta} \:\mathrm{d}\rho
      -   a(0) c^{\delta-1}\int_0^\infty \e{\i \rho}\rho^{-\delta}\:\mathrm{d}\rho\right| \\
    &\leq  \left|\int_0^1  \e{\i c\rho} (a(\rho)-a(0)) \rho^{-\delta} \:\mathrm{d}\rho \right| 
    + |a(0)| \left|\int_1^\infty \e{\i c\rho }\rho^{-\delta}\:\mathrm{d}\rho\right|  \\
    &\les c^{-1} \int_0^1  |((a(\rho)-a(0))\rho^{-\delta})'| \drho   
     + c^{-1} |a(0)| \left(1+\int_1^\infty \rho^{-1-\delta}\:\mathrm{d}\rho\right) \\
    &\les c^{-1} \left(|a(0)|+\|a'\|_\infty\right).
 \end{align*}
   The second part of (ii) is a direct consequence of the first part since all constants incorporated in
   $\les$ are independent of $a$ and $c$.
  \end{proof}
  
  The above result allows to determine the exact asymptotics in the singular case $\delta\in (0,1)$ and in
  particular some lower bound for large $c$ that we will need in the construction of counterexamples. Similar
  but slightly different results can be obtained for $\delta=0$.
  
  \medskip
  
  \begin{prop} \label{prop:asymptotics_delta=0}
    Let $a\in C^1([0,1])$. Then the following holds for all $c>0$
    \begin{align*}
       &\text{(i)}\qquad \left|\int_0^1  \e{\i c\rho} a(\rho)   \:\mathrm{d}\rho\right|   
       \les  (1+c)^{-1} (|a(0)|+\|a'\|_\infty).   \\ 
     &\text{(ii)}\qquad \left|\int_0^1  \e{\i c\rho} a(\rho)   \:\mathrm{d}\rho
      -  \frac{a(1)\e{\i c}-a(0)}{\i c} \right| \les  (1+c)^{-2}\left(|a(0)|+|a'(0)|+\|a''\|_\infty\right)
      \quad\text{if }a\in C^2([0,1]).
%    \intertext{In particular, there is $M>0$ independent of $a$ such that}     
%     &\qquad\qquad  \left| \int_0^1 \e{\i c\rho} a(\rho)\rho^{-\delta}\drho\right|
%      \gtrsim c^{\delta-1}|a(0)|
%      \qquad\text{for }c \geq M(1+\|a'\|_\infty |a(0)|^{-1})^{1/\delta}.
%      \end{align*}
%     \begin{align*}
%      \left|\int_0^1  \e{\i c\rho} a(\rho) \rho^{-\delta} \:\mathrm{d}\rho \right|  
% %      &\les  (1+c)^{\delta-1}\left(\|a'\|_\infty+|a(0)|\right).
   \end{align*} 
  \end{prop}
  \begin{proof}
    Part (i) is proved just as in the singular case, see Proposition~\ref{prop:asymptotics_delta>0}. 
    For $0<c\leq 1$ the estimate (ii) is trivial and for  $c>1$ we get via integration by parts and the
    estimate (i)
    \begin{align*}
    \left|\int_0^1  \e{\i c\rho} a(\rho)  \drho  -  \frac{a(1)\e{\i c}-a(0)}{\i c}  \right| 
    = c^{-1} \left|  \int_0^1  \e{\i  c\rho} a'(\rho)   \drho  \right|  
    \les c^{-2} (\|a''\|_\infty+|a'(0)|).  
   \end{align*}
  \end{proof}

  \medskip
  
  \noindent \textbf{Proof of Theorem~\ref{thm:Tlambdaalpha_d=1}:}
  We have to show that the estimate 
  \begin{equation}\label{eq:Testimate_d=1}
      \|T_{\lambda,\alpha} h\|_{L^q(\R)} \les (1+\lambda)^{2\alpha-\frac{2}{p}+\frac{2}{q}} \|h\|_{L^p(\R)}
  \end{equation}
  holds where $\alpha\in [0,1),\lambda\geq 0$ and $(p,q)\in \mathcal D_\alpha$, i.e.,
  $\frac{1}{p}-\frac{1}{q}\geq \alpha$, $\frac{1}{p}>\alpha$, $\frac{1}{q}<1-\alpha$. 
  Moreover, we will show that the range of exponents $p,q$ is optimal under these assumptions.
  
  \medskip
  
  For notational convenience we only consider the special case in Theorem~\ref{thm:Tlambdaalpha_d=1}
  where the annulus is $A=\{\xi\in\R : 1\leq |\xi|\leq 2\}$, so we consider (see~\eqref{eq:def_Tlambdaalpha})
  the operator
  \begin{equation*} 
     T_{\lambda,\alpha} h
      =  \cF_1^{-1}\left( 1_A(\cdot) \e{-\lambda\sqrt{|\cdot|^2-1}}
         (|\cdot|^2-1)^{-\alpha} m(|\cdot|) \cF_1 h(\cdot) \right).
  \end{equation*}
  
  \medskip
  
  We first present the comparatively easy proof for the case $1\leq p\leq 2\leq q\leq \infty,
  \frac{1}{p}-\frac{1}{q}\geq \alpha$ that only requires $m\in C([1,2])$. From the Hausdorff-Young
  inequality we get in the case $\frac{1}{p}-\frac{1}{q}> \alpha$
  \begin{align*}
    \|T_{\lambda,\alpha} h\|_{L^q(\R)} 
    &\les  \left\|\cF_1^{-1}\left( 1_A(\cdot) \e{-\lambda\sqrt{|\cdot|^2-1}}
      (|\cdot|^2-1)^{-\alpha} m(|\cdot|) \cF_1 h(\cdot) \right)\right\|_{L^q(\R)} \\
    &\les  \left\| \e{-\lambda\sqrt{|\cdot|^2-1}} (|\cdot|^2-1)^{-\alpha} m(|\cdot|) \cF_1 h(\cdot)
    \right\|_{L^{q'}(A)} \\
    &\les  \left\| \e{-\lambda\sqrt{|\cdot|^2-1}} (|\cdot|^2-1)^{-\alpha}
    m(|\cdot|)\right\|_{L^{\frac{pq}{q-p}}(A)} \|\cF_1 h \|_{L^{p'}(A)} \\
    &\les \|m\|_\infty \| \e{-\lambda\sqrt{|\cdot|^2-1}} (|\cdot|^2-1)^{-\alpha}
     \|_{L^{\frac{pq}{q-p}}(A)} \|\cF_1 h\|_{L^{p'}(\R)} \\
    &\les \|m\|_\infty \left( \int_1^2 \e{-\frac{pq}{q-p}\lambda\sqrt{r^2-1}} (r^2-1)^{-\frac{\alpha pq}{q-p}}
    \:\mathrm{d}r\right)^{\frac{q-p}{pq}} \|h\|_{L^p(\R)} \\
    &\les \|m\|_\infty \left( \int_0^1 \e{-\lambda\rho} \rho^{1-\frac{2\alpha pq}{q-p}}
    \:\mathrm{d}\rho\right)^{\frac{q-p}{pq}} \|h\|_{L^p(\R)} \\
    &\les \|m\|_\infty (1+\lambda)^{2\alpha-\frac{2}{p}+\frac{2}{q}} \|h\|_{L^p(\R)}.  
  \end{align*}
  The endpoint case $\frac{1}{p}-\frac{1}{q}=\alpha$ is achieved through complex interpolation. Since the
  procedure is almost the same as in the proof of Theorem~\ref{thm:Tlambdaalpha} below, we omit the details here and remark only that this strategy requires continuity of $m$.
    
  \medskip
  
  We continue with the proof of the full result under the assumption $m\in C^1([1,2])$. We use 
    \begin{align*}
      |T_{\lambda,\alpha}h(x)|
      &= \left|\cF_1^{-1}\left( 1_A(\cdot) \e{-\lambda\sqrt{|\cdot|^2-1}}(|\cdot|^2-1)^{-\alpha} m(|\cdot|)
          \cF_1 h\right)(x) \right|    \\
      &= \frac{1}{\sqrt{2\pi}}  \left| \int_A \e{\i\xi x} \e{-\lambda\sqrt{|\xi|^2-1}} (|\xi|^2-1)^{-\alpha}
      m(|\xi|) \cF_1 h(\xi)\:\mathrm{d}\xi \right|  \\
      &= \left| \int_{\R} K_\lambda(x-y) h(y)  \:\mathrm{d}y \right|  \\
    \text{where}\qquad
      K_\lambda(z)
      &:= \frac{1}{2\pi} \int_A \e{\i\xi z}  
      (|\xi|^2-1)^{-\alpha} m(|\xi|) \e{-\lambda\sqrt{|\xi|^2-1}} \dxi.   
    \end{align*}
    From this identity we get in the case $p\neq 1,q\neq \infty$
    \begin{align*}
      \|T_{\lambda,\alpha} h\|_{L^q(\R)}
      \les \|   K_\lambda \ast h \|_{L^q(\R)} 
      \les \|K_\lambda\|_{L^{\frac{pq}{pq+p-q},\infty}(\R)} \|h\|_{L^p(\R)}  
    \end{align*}
    so that we have to show
    \begin{equation}\label{eq:Klambda_weak_d=1}
      \|K_\lambda\|_{L^{\frac{pq}{pq+p-q},\infty}(\R)} 
      \les (1+\lambda)^{2\alpha-\frac{2}{p}+\frac{2}{q}} \qquad \text{if }p\neq 1,\;q\neq \infty.
    \end{equation}
    Similarly, in order to cover the cases $p=1$ or $q=\infty$ as well, we need to prove
    \begin{equation}\label{eq:Klambda_strong_d=1}
      \|K_\lambda\|_{L^{\frac{pq}{pq+p-q}}(\R)} 
      \les (1+\lambda)^{2\alpha-\frac{2}{p}+\frac{2}{q}} \qquad \text{if }p=1 \text{ or }q=\infty.
    \end{equation}
    
    \medskip
   
   The proof of~\eqref{eq:Klambda_weak_d=1},\eqref{eq:Klambda_strong_d=1} is based on 
   pointwise estimates for the kernel function $K_\lambda$. For $|z|\leq 1+\lambda^2$  we will use
   \begin{align*} 
     |K_\lambda(z)|
     \les   \int_1^2 \e{-\lambda \sqrt{r^2-1}} (r^2-1)^{-\alpha}  \dr     
     \les  \int_0^{\sqrt 3} \e{-\lambda \rho} \rho^{1-2\alpha}  \drho  
     \les (1+\lambda)^{2\alpha-2}. 
   \end{align*}
   For  $|z|\geq 1+\lambda^2$  we estimate the kernel with the aid of
   oscillatory integral theory that uses $m\in C^1([1,2])$. Proposition~\ref{prop:asymptotics_delta>0}~(i)
   and Proposition~\ref{prop:asymptotics_delta=0}~(i) imply
    \begin{align}  \label{eq:d=1_Klambda_boundII}
      \begin{aligned}
      |K_\lambda(z)|
      &\les     \left| \int_0^1 \e{\i \rho z} \rho^{-\alpha} 
       (2+\rho)^{-\alpha} m(1+\rho)\e{-\lambda\sqrt{\rho^2+2\rho}} 
       \:\mathrm{d}\rho\right|     \\
       &\!\!\!\!\!\stackrel{\text{Prop. }\ref{prop:Stein}}\les     \int_0^{|z|^{-1}} 
        \rho^{-\alpha} (2+\rho)^{-\alpha} |m(1+\rho)|\e{-\lambda\sqrt{\rho^2+2\rho}}
       \:\mathrm{d}\rho  \\
       &\qquad + \frac{1}{|z|} \left( |z|^{\alpha} + \int_{|z|^{-1}}^1 \left| \frac{\mathrm{d}}{\mathrm{d}\rho} \left(
         \rho^{-\alpha} (2+\rho)^{-\alpha} m(1+\rho)\e{-\lambda\sqrt{\rho^2+2\rho}}
       \right)\right|  \:\mathrm{d}\rho \right)     \\
       &\les     \int_0^{|z|^{-1}} 
        \rho^{-\alpha}  \e{-\lambda\sqrt{\rho}}
       \:\mathrm{d}\rho 
       + |z|^{\alpha-1} \\
      &\qquad + \frac{1}{|z|} \int_{|z|^{-1}}^1 \rho^{-\alpha-1}
       \e{-\lambda\sqrt{\rho}}\:\mathrm{d}\rho + \frac{\lambda}{|z|} \int_{|z|^{-1}}^1 \rho^{-\alpha-\frac{1}{2}} \e{-\lambda\sqrt{\rho}}
		   \:\mathrm{d}\rho \\
       &\les  \lambda^{2\alpha-2}    \int_0^{\lambda |z|^{-1/2}} t^{1-2\alpha} \e{-t}   \:\mathrm{d}t 
		  + |z|^{\alpha-1} \\
	  & \qquad + \frac{\lambda^{2\alpha}}{|z|} \left( \int_{\lambda |z|^{-1/2}}^\lambda 
       t^{-2\alpha}  \e{-t}  + t^{-1-2\alpha} \e{-t} 
       \:\mathrm{d}t \right)     \\
       &\les  \lambda^{2\alpha-2}    \int_0^{\lambda |z|^{-1/2}} t^{1-2\alpha} \e{-t}   \:\mathrm{d}t 
		  + |z|^{\alpha-1} + \frac{\lambda^{2\alpha}}{|z|}  \int_{\lambda |z|^{-1/2}}^\infty 
        t^{-1-2\alpha} \e{-t} \:\mathrm{d}t      \\
       &\les  \lambda^{2\alpha-2}    \left( \lambda |z|^{-1/2} \right)^{2-2\alpha}  
        + |z|^{\alpha-1} + \frac{\lambda^{2\alpha}}{|z|}  \left( \left(\lambda
        |z|^{-1/2}\right)^{-2\alpha}+1\right) \\
       &\les |z|^{\alpha-1}.  
    \end{aligned}
    \end{align}
%     Since our assumtion $\frac{1}{p}-\frac{1}{q}\geq \alpha$ implies $\frac{pq}{pq+p-q}(\alpha-1)\leq -1$, the
%     above estimate proves \eqref{eq:Klambda_weak_d=1} for $\lambda\leq 1$. 
    Making use of $\frac{pq}{pq+p-q} \geq \frac{1}{1-\alpha}$ (due to  $\frac{1}{p}-\frac{1}{q}\geq \alpha$)
    we get
    \begin{align*}
      \|K_\lambda\|_{L^{\frac{pq}{pq+p-q},\infty}(\R)}
       &\les  (1+\lambda)^{2\alpha-2} \|1\|_{L^{\frac{pq}{pq+p-q},\infty}([0,1+\lambda^2])} 
       +    \| |\cdot|^{\alpha-1}\|_{L^{\frac{pq}{pq+p-q},\infty}([1+\lambda^2,\infty))} \\  
       &\les  (1+\lambda)^{2\alpha-\frac{2}{p}+\frac{2}{q}}
     \intertext{  
    We conclude that \eqref{eq:Klambda_weak_d=1} holds. Along the same lines we find }
      \|K_\lambda\|_{L^{\frac{pq}{pq+p-q}}(\R)}
      &\les   (1+\lambda)^{2\alpha-\frac{2}{p}+\frac{2}{q}} \qquad\qquad\text{if
      }p=1 \text{ or }q=\infty
    \end{align*} 
    because then $\frac{pq}{pq+p-q} > \frac{1}{1-\alpha}$ by assumption. 
    So the sufficiency part of Theorem~\ref{thm:Tlambdaalpha_d=1} is proved.
    
    \medskip
  
    For the construction of a counterexample we assume $m\equiv 1, \lambda=0$  
    as well as $\frac{1}{p}-\frac{1}{q}<\alpha$ and $p\neq 1,q\neq \infty$. We want to show
    that~\eqref{eq:Testimate_d=1} does not hold in this case. To this end we choose $\beta$
    according to 
    \begin{equation} \label{eq:choice_beta_d=1}
      \max\left\{1-\alpha-\frac{1}{q},0\right\} < \beta < 1-\frac{1}{p}.
    \end{equation}
    Then the function $f:= \sqrt{2\pi} \cF_1^{-1}(1_{[1,2]}(\cdot)(|\cdot|^2-1)^{-\beta})$ belongs to
    $L^p(\R)$ because of
    \begin{align*}
      |f(x)| 
      = \left|\int_1^2 \e{\i x\xi} (|\xi|^2-1)^{-\beta}\,\mathrm{d}\xi\right| 
      = \left| \int_0^1 \e{\i x\rho} \rho^{-\beta} (2+\rho)^{-\beta}\,\mathrm{d}\rho\right|
      \les (1+|x|)^{\beta-1},  
    \end{align*}
    see Proposition~\ref{prop:asymptotics_delta>0}~(i). On the other hand,
    Proposition~\ref{prop:asymptotics_delta>0}~(ii) gives in the case $\alpha+\beta <1$
    \begin{align*}
      \left|T_{0,\alpha}f(x)\right| 
      %= \left|\int_1^2 \e{\i x\xi} (|\xi|^2-1)^{-\beta-\alpha}\,\mathrm{d}\xi\right| 
      = \left|\int_0^1 \e{\i x\rho} \rho^{-\alpha-\beta} (2+\rho)^{-\alpha-\beta}\,\mathrm{d}\rho \right|
      \gtrsim |x|^{\alpha+\beta-1} \quad\text{as }|x|\to\infty.  
    \end{align*}    
    Since our choice for $\beta$ from~\eqref{eq:choice_beta_d=1} implies $q(\alpha+\beta-1)\geq -1$, this
    estimate gives $T_{0,\alpha}f \notin L^q(\R)$.
    In the case $\alpha+\beta \geq 1$ we slightly modify the counterexample and define 
    $f_\eps:=\cF_1^{-1}( 1_{[1+\eps,2]}(\cdot)(|\cdot|^2-1)^{-\beta})$. Then the sequence
    $(f_\eps)$ is bounded in $L^p(\R)$ by the Hausdorff-Young inequality while 
    $|T_{0,\alpha}f_\eps(x)|\to +\infty$ uniformly on a small neighbourhood of $x=0$. 
    Indeed,  
    \begin{align*}
      \inf_{|x|\leq \pi/8} \left|T_{0,\alpha}f_\eps (x)\right| 
      &= \inf_{|x|\leq \pi/8} \left|\int_{1+\eps}^2 \e{\i x\xi} (|\xi|^2-1)^{-\beta-\alpha}\,\mathrm{d}\xi\right|  \\
      &\geq  \inf_{|x|\leq \pi/8} \left|\int_{1+\eps}^2 \cos(x\xi) (|\xi|^2-1)^{-\beta-\alpha}\,\mathrm{d}\xi\right|
      \\
      &\geq \cos(\pi/4)  \int_{1+\eps}^2   (|\xi|^2-1)^{-\beta-\alpha}\,\mathrm{d}\xi  
      \nearrow \infty \quad\text{as }\eps\to 0.  
    \end{align*}
    This shows that $T_{0,\alpha}$ is unbounded for $\frac{1}{p}-\frac{1}{q}<\alpha$ and
    $p\neq 1,q\neq \infty$.  
    
    \medskip
    
    It remains to show that~\eqref{eq:Testimate_d=1} does not hold either for 
    $p=\frac{1}{\alpha},q=\infty$ or $p=1,q=\frac{1}{1-\alpha}$ and $\alpha\in [0,1)$. By duality, it
    suffices to disprove \eqref{eq:Testimate_d=1} in the former case. 
%    \r{Indeed, having shown that, the
%    $L^1-L^{\frac{1}{1-\alpha}}$-boundedness of $T$ would imply the
%    $L^{\frac{1}{\alpha}}-L^\infty$-boundedness, 
%    which would contradict the presence of the counterexample.}
     This example is constructed as follows. We set for $k \in \N$
    $$
      f_k(y):= \ln(k+1)^{-\alpha} 1_{[1,k+1]}(y) y^{-\alpha}\e{\i y}.
    $$ 
    Then $\|f_k\|_{L^p(\R)}=1$ and
    \begin{align*}
      \|T_{0,\alpha}f_k\|_{L^\infty(\R)}
      \geq |T_{0,\alpha}f_k(0)|  
      = \frac{1}{2\pi} \ln(k+1)^{-\alpha} \left| \int_1^{k+1} y^{-\alpha} \e{\i y} \left( \int_1^2
      \e{-\i\xi y}(|\xi|^2-1)^{-\alpha} \dxi \right) \dy \right|
   \end{align*}
   In the case $\alpha=0$ this implies
   \begin{align*}
      \|T_{0,0}f_k\|_{L^\infty(\R)}
      &\geq  \frac{1}{2\pi}  \left| \int_1^{k+1}  \e{\i y} \left( \int_1^2
      \e{-\i\xi y}  \dxi \right) \dy \right| \\
      &= \frac{1}{2\pi}  \left| \int_1^{k+1}  \e{\i y} \cdot \frac{\e{-2\i y}-\e{-\i y}}{-\i y}  \dy \right| \\
      &= \frac{1}{2\pi}  \left| \int_1^{k+1}  \frac{\e{-\i y}}{y} - \frac{1}{y}    \dy \right|  \\
      &\geq \frac{\ln(k+1)}{2\pi}   - \frac{1}{2\pi} \left| \int_1^{k+1}  \frac{\e{-\i y}}{y}   \dy \right|,
   \end{align*}
   which tends to $+\infty$ as $k \to \infty$.
   This proves the unboundedness in the case $\alpha=0$, i.e., for $p=q=\infty$. In the case $\alpha\in (0,1)$
   we get from Proposition~\ref{prop:asymptotics_delta>0}~(ii)   
   \begin{align*}
     \lim_{y\to\infty} y^{1-\alpha} \e{\i y} \int_1^2 \e{-\i\xi y} (|\xi|^2-1)^{-\alpha}\dxi 
     &=  \lim_{y\to\infty} y^{1-\alpha} \int_0^1 \e{-\i\rho y}\rho^{-\alpha} (2+\rho)^{-\alpha}\drho \\
     &=  2^{-\alpha} \int_0^\infty \e{-\i \rho} \rho^{-\alpha} \drho 
     \\
     &=:\mu \in\C\sm\{0\}. 
   \end{align*}   
   Hence, for $k_0 \in \N$ sufficiently large and all $k\geq k_0$ we have   
   \begin{align*}
      \|T_{0,\alpha}f_k\|_{L^\infty(\R)}
      &\geq |T_{0,\alpha}f_k(0)| \\ 
      &=\frac{1}{2\pi} \ln(k+1)^{-\alpha} \left| \int_1^{k+1} y^{-\alpha} \e{\i y} \left( \int_1^2 \e{-\i\xi
      y}(|\xi|^2-1)^{-\alpha}\:\mathrm{d}\xi \right) \:\mathrm{d}y \right| \\
      &=  \frac{1}{2\pi} \ln(k+1)^{-\alpha} \left| \int_{k_0}^{k+1} y^{-1}\cdot y^{1-\alpha} \e{\i y} \left( \int_1^2 \e{-\i\xi
      y}(|\xi|^2-1)^{-\alpha}\:\mathrm{d}\xi \right) \:\mathrm{d}y \right| + o(1) \\
      &\geq \frac{|\mu|}{4\pi} \ln(k+1)^{-\alpha} \int_{k_0 }^{k+1} y^{-1}\dy + o(1) \\
      &= \frac{|\mu|}{4\pi} \ln(k+1)^{1-\alpha}  + o(1),  
   \end{align*}  
   which tends to $+\infty$ as $k \to \infty$.
   Hence, the operator $T_{0,\alpha}:L^p(\R)\to L^q(\R)$ is unbounded for 
   $p=\frac{1}{\alpha},q=\infty$ and $\alpha\in (0,1)$, which is all we had to show.
   \qed

\section{Proof of Theorem~\ref{thm:Tlambdaalpha}} \label{sec:ProofTlambda}

Theorem~\ref{thm:Tlambdaalpha} is proved with the aid of Stein's Interpolation Theorem
\cite[Theorem~1]{Stein_Interpolation} for holomorphic families of operators. So we have to
estimate the operators $T_{\lambda,\alpha}$ defined in~\eqref{eq:def_Tlambdaalpha}. 
We first recall our estimates for the operators $S_\lambda$ from~\eqref{eq:def_Slambda} that will
provide the desired bounds in the endpoint case $\alpha=0$. Choosing $s=2$ in Theorem~\ref{thm:Slambdaalpha}
we get the following.     

\begin{prop} \label{prop:Slambda0}
  Let $d\in\N,d\geq 2$, $0<a<b<\infty$ and $m\in L^\infty([a,b])$. Then we have for all $\lambda\geq 0$
  and all $p\in [1,2]$
  \begin{align*}%\label{eq:Slambda0_estimate}
    %\begin{aligned}
      \|S_{\lambda} h\|_{L^2(A)}  
      &\les (1+\lambda)^{- 1+(\frac{d+3}{2}-\frac{d+1}{p})_+} \|h\|_{L^p(\R^d)}, \\
      \|S_{\lambda}^* g\|_{L^{p'}(\R^d)} 
      &\les (1+\lambda)^{- 1+(\frac{d+3}{2}-\frac{d+1}{p})_+} \|g\|_{L^2(A)}.  
    %\end{aligned}
  \end{align*}
\end{prop}

As a consequence we obtain the following result.

\begin{prop} \label{prop:Tlambda0}
  Let $d\in\N,d\geq 2$, $0<a<b<\infty$ and $m\in L^\infty([a,b])$. Then we have for all $\lambda\geq 0$
  and all $p\in [1,2],q\in [2,\infty]$
  \begin{align*}%\label{eq:Tlambda0_estimate}
    %\begin{aligned}
      \|T_{\lambda,0} h\|_{L^q(\R^d)}  
      \les (1+\lambda)^{- 2+(\frac{d+3}{2}-\frac{d+1}{p})_+ + (\frac{d+3}{2}-\frac{d+1}{q'})_+}  \|h\|_{L^p(\R^d)}.
    %\end{aligned}
  \end{align*}
\end{prop}
\begin{proof}
 We may assume that $m$ is real-valued nonnegative, otherwise we split the operator into the sum of four
  such operators according to 
  $m= \mathfrak m_1^2- \mathfrak m_2^2+\i(\mathfrak m_3^2-\mathfrak m_4^2)$. 
  But then we have
  $T_{\lambda,0}=S_{\lambda}^* S_{\lambda}$ where $m$ in the definition of
  $S_{\lambda},S_{\lambda}$ is replaced by $\sqrt{m}$ and thus the claim follows from
  Proposition~\ref{prop:Slambda0}.
\end{proof}

Next we use these estimates in the endpoint case
$\alpha=0$ for the analysis of $T_{\lambda,\alpha}$ from~\eqref{eq:def_Tlambdaalpha} with $\alpha\in
(0,1)$. Up to an $\alpha$-dependent prefactor, these operators may be embedded
into the family of operators
\begin{equation} \label{eq:def_Tlambdas}
   \mathcal T_{\lambda,s} h
  :=  \frac{\e{(1-s)^2}}{\Gamma(1-s)} \cF_d^{-1}\left( 1_A(\cdot) \e{-\lambda\sqrt{|\cdot|^2-a^2}}
  (|\cdot|^2-a^2)^{-s} m(|\cdot|) \cF_d h(\cdot) \right).
\end{equation} 
A priori, these operators are well-defined for Schwartz functions $h:\R^d\to\C$ and 
$s\in\C$ with $0\leq \Real(s)<1$. We assume $\lambda\geq 0$ and $m\in C([a,b])$. Since we are
going to apply Stein's Interpolation Theorem (Theorem~1 in~\cite{Stein_Interpolation}) to the family
$(\mathcal T_{\lambda,\sigma s})_{s\in S}$ where $S:= \{ s\in \C: 0\leq \Real(s)\leq 1\}$ and $\sigma\in
[0,1]$ (including the endpoint case $\sigma=1$), we need to extend the operators from~\eqref{eq:def_Tlambdas}
to the line $\Real(s)=1$ in a continuous way. Only for this reason we will temporarily assume  
$m\in C^1([a,b])$, but we will see that this extra assumption is actually not necessary. The extension is
based on the representation
\begin{align*}
   (\mathcal T_{\lambda,s} h)(x)
   &=  \frac{\e{(1-s)^2}}{\Gamma(1-s)} \int_a^b \e{-\lambda\sqrt{r^2-a^2}} (r^2-a^2)^{-s} m(r)
      \cF_d^{-1}( \cF_d h \:\mathrm{d}\sigma_r)(x)\:\mathrm{d}r \\
   &= (1-s) \int_a^b (r-a)^{-s} (\mathcal A_{\lambda,s}(r)h)(x) \:\mathrm{d}r
   \quad\text{where } \\
   (\mathcal A_{\lambda,s}(r)h)(x) &:=   \frac{\e{(1-s)^2}}{\Gamma(2-s)}  
   \e{-\lambda\sqrt{r^2-a^2}} (r+a)^{-s}m(r) \cF_d^{-1}(\cF_d h \:\mathrm{d}\sigma_r)(x). 
\end{align*}
Integration by parts motivates the definition
\begin{align*}
  (\mathcal T_{\lambda,s} h)(x)
   &=  (b-a)^{1-s} (\mathcal A_{\lambda,s}(b)h)(x) - \int_a^b (r-a)^{1-s} (\mathcal A_{\lambda,s}'(r)h)(x)
   \dr \qquad\text{if } \Real(s)=1.
\end{align*}
Notice that this expression is well-defined for Schwartz functions $h:\R^d\to\C$ 
(due to $m\in C^1([a,b])$) and we have 
$$
  \mathcal T_{\lambda,1}h
  = \mathcal A_{\lambda,1}(a)h 
  =  (2a)^{-1}m(a) \cF_d^{-1}(\cF_d h \:\mathrm{d}\sigma_a).
$$
%which, up to a multiplicative factor and rescaling, is the imaginary part of the operator
%investigated in~\cite[Theorem~6]{Gut_Nontrivial}.
In order to apply the Interpolation Theorem, we need to check that $(\mathcal T_{\lambda,s})_{s\in S}$ is an
analytic family of operators in the sense of~\cite[p.483]{Stein_Interpolation}.

\begin{prop} \label{prop:admissiblefamily}
  For all Schwartz functions $h_1,h_2:\R^d\to\C$ the map 
  $
   s\mapsto \int_{\R^d} h_1 (\mathcal T_{\lambda,s}h_2) \:\mathrm{d}x
  $
  is holomorphic in $\mathring S$ and continuous on $S$.  
\end{prop}

%\r{Should we give the proof for that? Stein actually requires this for step functions!} \\
Proposition~\ref{prop:admissiblefamily} implies that for all $\sigma\in [0,1]$ the family
$(\mathcal T_{\lambda,\sigma s})_{s\in S}$ is admissible for Stein's Interpolation Theorem. 
Notice that the original version requires Proposition~\ref{prop:admissiblefamily} to hold for step functions,
but actually any dense family of functions can be chosen.
We use this fact
in order to show that $T_{\lambda,\alpha}$, which is an $\alpha$-dependent multiple of $\mathcal
T_{\lambda,\alpha}$, is a bounded operator from $L^p(\R^d)$ to $L^q(\R^d)$ whenever $\alpha\in (0,1)$ and $(p,q)\in D_\alpha$ where 
\begin{align*}
  D_\alpha
  &= \left\{ (p,q)\in [1,\infty]^2: \frac{1}{p}>\frac{1}{2}+\frac{\alpha}{2d},\;
   \frac{1}{q}<\frac{1}{2}-\frac{\alpha}{2d},\; \frac{1}{p}-\frac{1}{q}\geq \frac{2\alpha}{d+1}\right\},
\end{align*}
cf.~\eqref{eq:def_D}.  An estimate of the corresponding mapping
constant will then provide the result.

\medskip

\noindent \textbf{Proof of Theorem~\ref{thm:Tlambdaalpha}:} For notational convenience we only discuss
$a=1,b=2$. As explained earlier, our proof is based on complex interpolation. We temporarily assume
$m\in C^1([1,2])$ in order to make use of Proposition~\ref{prop:admissiblefamily} that is needed for
Stein's Interpolation Theorem. On the other hand, our estimates will only depend on the $L^\infty$-norm of $m$
so that all results will persist for $m$ belonging to the completion of $C^1([1,2])$ with respect to this
norm, namely for $m\in C([1,2])$. 
%For that reason it is sufficient to assume $m\in C^1([1,2])$ in the following.

\medskip
  
  We start with recalling the estimates for the endpoint $\alpha=0$. From Proposition~\ref{prop:Tlambda0} we
  deduce 
  \begin{align} \label{eq:interpolation_alpha0}
    \|\mathcal T_{\lambda,0}f\|_{L^{q_1}(\R^d)}
    = \|T_{\lambda,0}f\|_{L^{q_1}(\R^d)}
    \les (1+\lambda)^{- 2  +  (\frac{d+3}{2}-\frac{d+1}{q_1'})_++(\frac{d+3}{2}-\frac{d+1}{p_1})_+}     
    \|f\|_{L^{p_1}(\R^d)} 
  \end{align}
  whenever $1\leq p_1\leq 2\leq q_1\leq \infty$. Those already yield the claim 
  for $\alpha=0$ so that we may assume $\alpha\in (0,1)$ in the following.   
  On the other hand, for exponents $p_2,q_2$ satisfying 
  $\frac{1}{p_2}>\frac{d+1}{2d},\frac{1}{q_2}<\frac{d-1}{2d},\frac{1}{p_2}-\frac{1}{q_2}\geq
  \frac{2}{d+1}$ we get for any $s\in S$ with $0\leq \Real(s)<1$ from Minkowski's
  inequality in integral form and Corollary~\ref{cor:Gutierrez}
  \begin{align} \label{eq:interpolation_alphas}
    \begin{aligned}
    \|\mathcal T_{\lambda,s}f\|_{L^{q_2}(\R^d)}
    &= \left\|\frac{\e{(1-s)^2}}{\Gamma(1-s)} \int_1^2
    \e{-\lambda\sqrt{r^2-1}}(r^2-1)^{-s}m(r)\cF_d^{-1}\left( (\cF_df \:\mathrm{d}\sigma_r)(\cdot)
    \:\mathrm{d}r \right)\right\|_{L^{q_2}(\R^d)} \\
    &\leq \left|\frac{\e{(1-s)^2}}{\Gamma(1-s)}\right| \int_1^2  \e{-\lambda\sqrt{r^2-1}}
    (r^2-1)^{-\Real(s)} m(r) \left\|\cF_d^{-1}\left(\cF_d f \:\mathrm{d}\sigma_r\right) \right\|_{L^{q_2}(\R^d)} \:\mathrm{d}r \\
    &\leq \left|\frac{\e{(1-s)^2}}{\Gamma(1-s)}\right|\int_1^2  \e{-\lambda\sqrt{r^2-1}}
    (r^2-1)^{-\Real(s)} m(r) r^{-1+\frac{d}{p_2}-\frac{d}{q_2}} \|f\|_{L^{p_2}(\R^d)} \:\mathrm{d}r    \\
    &\les \left|\frac{\e{(1-s)^2}}{\Gamma(1-s)}\right| \|m\|_\infty  \left(\int_0^{\sqrt 3}  \e{-\lambda\rho}
    \rho^{1-2\Real(s)}    \:\mathrm{d}\rho \right) \|f\|_{L^{p_2}(\R^d)}  \\
    &\les \left|\frac{\e{(1-s)^2}}{\Gamma(1-s)}\right| \|m\|_\infty |1-\Real(s)| (1+\lambda)^{2\Real(s)-2}
    \|f\|_{L^{p_2}(\R^d)}  \\
    &\les \|m\|_\infty (1+\lambda)^{2\Real(s)-2} \|f\|_{L^{p_2}(\R^d)}.
  \end{aligned}
  \end{align}
  By our choice of the prefactor, which is adapted from \cite[p.381]{Stein_Harmonic}, the above estimate is
  uniform with respect to $s\in S$ such that $0\leq \Real(s)<1$. Moreover, as announced earlier, it only
  depends on the $L^\infty$-norm of $m$. Hence, the
  continuity property from Proposition~\ref{prop:admissiblefamily} implies that the estimate persists on the closure of this set,
  namely on the whole strip $S$. This is a consequence of the Uniform Boundedness Principle. From
  Proposition~\ref{prop:admissiblefamily},~\eqref{eq:interpolation_alpha0},~\eqref{eq:interpolation_alphas} we
  infer that, for any given $\sigma\in [0,1]$,  $(\mathcal T_{\lambda,\sigma s})_{s\in\Ss}$ is a holomorphic
  family of operators of admissible growth in the sense of \cite{Stein_Interpolation} so that Stein's
  Interpolation Theorem applies. 
  
  \medskip
  
  We consider three different regimes of exponents $(p,q)\in D_\alpha$, namely
  \begin{itemize}
    \item[(a)] $(p,q)\in \mathcal D$, i.e.,  $\frac{1}{p}>\frac{d+1}{2d}$,\;
    $\frac{1}{q}<\frac{d-1}{2d}$,\;$\frac{1}{p}-\frac{1}{q}\geq \frac{2}{d+1}$,
    \item[(b)] $\frac{1}{p}-\frac{1}{q}< \min\left\{ \frac{2}{d+1}, \frac{2d}{d+1} (2m_{p,q}-1)\right\}$,
	\item[(c)]  $m_{p,q}\leq \frac{d+1}{2d}$ and  $\frac{1}{p}-\frac{1}{q} \geq \frac{2d}{d+1}(2m_{p,q}-1)$.
  \end{itemize} 
  Here, $m_{p,q}:=\min\{\frac{1}{p},1-\frac{1}{q}\}$. 
  %For an illustration of the situation we refer to \r{Figure TODO}.
  %
  %\bigskip
  First, for $(p,q)$ as in (a) we do not need interpolation to conclude. Indeed, the above estimate implies
  $$
    \|\mathcal T_{\lambda,\alpha}f\|_{L^q(\R^d)}
    \les (1+\lambda)^\gamma \|f\|_{L^p(\R^d)}
    \qquad
    \text{where } \gamma = 2\alpha-2 <  2\alpha-2 + \frac{1}{p}-\frac{1}{q}
  $$ 
  Hence, \eqref{eq:gamma_conditions} holds and the claim is proved for such exponents. 
 For exponents $(p,q)$ as in (b) or (c) we use interpolation. Having the above conditions on $p_1,q_1,p_2,q_2$
 in mind, Stein's Interpolation Theorem gives
  \begin{align*}
    &\| \mathcal T_{\lambda,\theta\sigma} f\|_{L^q(\R^d)}
    \les (1+\lambda)^{2\theta\sigma-2+(1-\theta)(\frac{d+3}{2}-\frac{d+1}{p_1})_+
    +(1-\theta)(\frac{d+3}{2}-\frac{d+1}{q_1'})_+} \|f\|_{L^p(\R^d)},   \\
    &\text{where}\quad \frac{1}{p}=\frac{1-\theta}{p_1}+\frac{\theta}{p_2},\quad
    \frac{1}{q}=\frac{1-\theta}{q_1}+\frac{\theta}{q_2},\quad \theta\in [0,1),\;\sigma \in [0,1].
  \end{align*}
  Being interested in $\theta\sigma=\alpha$  we thus obtain $(\sigma:=\alpha/\theta)$
  \begin{align}\label{eq:interpolation}
    \begin{aligned}
    &\|\mathcal T_{\lambda,\alpha}f\|_{L^q(\R^d)}
    \les (1+\lambda)^{2\alpha-2+(1-\theta)(\frac{d+3}{2}-\frac{d+1}{p_1})_+
    +(1-\theta)(\frac{d+3}{2}-\frac{d+1}{q_1'})_+} \|f\|_{L^p(\R^d)},  \\
    &\text{where}\quad  \frac{1}{p}=\frac{1-\theta}{p_1}+\frac{\theta}{p_2},\qquad
    \frac{1}{q}=\frac{1-\theta}{q_1}+\frac{\theta}{q_2},\qquad \theta\in [\alpha,1), \\
    &1\leq p_1\leq 2\leq q_1\leq \infty,\qquad  \frac{1}{p_2}-\frac{1}{q_2}\geq
    \frac{2}{d+1},\qquad 1\geq \frac{1}{p_2},\frac{1}{q_2'}>\frac{d+1}{2d}.
  \end{aligned}
  \end{align} 
  In the case (b) one can check that that the choice 
  \begin{align*}
      \theta=\frac{d+1}{2}\left(\frac{1}{p}-\frac{1}{q}\right),\quad 
      p_1=q_1=2,\quad \frac{1}{p_2} =
      \frac{1}{2}+\frac{2}{d+1}\frac{\frac{1}{p}-\frac{1}{2}}{\frac{1}{p}-\frac{1}{q}},\quad
      \frac{1}{q_2} = \frac{1}{2}-\frac{2}{d+1}\frac{\frac{1}{2}-\frac{1}{q}}{\frac{1}{p}-\frac{1}{q}}.
  \end{align*}  
%   \r{VERIFICATION: $\frac{1}{p_2}-\frac{1}{q_2}=\frac{2}{d+1}$ and 
%   \begin{align*}
%     p_2\geq 1,q_2\leq \infty
%     &\Leftrightarrow \frac{d-3}{p}-\frac{d+1}{q} \geq -2,  \frac{d+1}{p}-\frac{d-3}{q} \geq 2\\
%     &\Leftrightarrow \frac{d+1}{p}-\frac{d+1}{q} \geq -2+4M_{p,q} \\
%     &\Leftrightarrow \frac{1}{p}-\frac{1}{q}\geq \frac{2}{d+1}(2M_{p,q}-1) \\
%     &\Leftrightarrow \text{true for }1\leq p\leq 2\leq q\leq \infty
%   \end{align*}
%     \begin{align*}
%     \frac{1}{p_2}> \frac{d+1}{2d}, \frac{1}{q_2}<\frac{d-1}{2d}, 
%     &\Leftrightarrow \frac{3d-1}{p}+\frac{d+1}{q} >2d,\quad \frac{3d-1}{q}+\frac{d+1}{p}<2d \\
%     &\Leftrightarrow  \frac{1}{p}-\frac{1}{q}>\frac{2d}{d-1}|\frac{1}{p}+\frac{1}{q}-1|
%   \end{align*}
%   }  
  is admissible for~\eqref{eq:interpolation} and leads to the uuper bound for the operator norm 
   $$
     (1+\lambda)^\gamma 
     \quad
     \text{where }
     \gamma = 2\alpha- 2  +  2(1-\theta)   = 2\alpha- \frac{d+1}{p}+\frac{d+1}{q}.
   $$
   In particular, assuming additionally $\frac{1}{p}-\frac{1}{q}\geq \frac{2}{d+2}$ as in the Theorem, one
   finds $\gamma\leq 2\alpha-2 + \frac{1}{p}-\frac{1}{q}$. Given that $p\neq 1$ and $q\neq \infty$ we
   conclude that~\eqref{eq:gamma_conditions} holds under this assumption and the claim is proved for such exponents.
   
   \bigskip
   
   It remains to consider exponents $(p,q)$ as in (c). In that case we define $\theta_\eps:= 
   2dm_{p,q} -d -\eps$ for small $\eps>0$. In the case $m_{p,q}=\frac{1}{p}$ one chooses 
   \begin{align*}
    \theta = \theta_\eps,\quad
    p_1=2,\quad
    q_1 = \frac{1-\theta_\eps}{(\frac{1}{2}-\frac{1}{p}+\frac{1}{q}-\frac{\theta_\eps(d-3)}{2(d+1)})_+},\quad  
    p_2 = \frac{\theta_\eps}{\frac{1}{p}-\frac{1-\theta_\eps}{2}},\quad
    q_2 
    %= \frac{\theta}{\frac{1}{q}-  (\frac{1}{2}-\frac{1}{p}+\frac{1}{q}-\frac{\theta(d-3)}{2(d+1)}) }
    = \frac{\theta_\eps}{\frac{1}{q} -  \frac{1-\theta_\eps}{q_1}} 
  \end{align*}
  and in the case $m_{p,q}=1-\frac{1}{q}$ one takes
   \begin{align*}
    \theta = \theta_\eps,\quad
    q_1=2,\quad
    p_1=\left(\frac{1-\theta_\eps}{(\frac{1}{2}-\frac{1}{p}+\frac{1}{q}-\frac{\theta_\eps(d-3)}{2(d+1)})_+}\right)',
    \quad q_2 = \frac{\theta_\eps}{\frac{1}{q}-\frac{1-\theta_\eps}{2}},\quad 
    p_2 = \frac{\theta_\eps}{\frac{1}{p} -  \frac{1-\theta_\eps}{p_1}}.
  \end{align*}
  A lengthy computation reveals that these choices are admissible for~\eqref{eq:interpolation}
  and the bound for the operator norm is $(1+\lambda)^\gamma$ where 
  $$
    \gamma =
    \begin{cases}
      2\alpha+ \frac{d+1}{p}-\frac{d+1}{q} &, \text{if }
      2dm_{p,q}-\frac{d+1}{p}+\frac{d+1}{q} \geq d-1 \\
      2\alpha + 1-d+2dm_{p,q}+\frac{d+1}{p}-\frac{d+1}{q} + \eps &, \text{if }
      2dm_{p,q}-\frac{d+1}{p}+\frac{d+1}{q} < d-1. 
    \end{cases}
  $$
  Under the additional assumption $\frac{1}{p}-\frac{1}{q}\geq \frac{2}{d+2}$ from the Theorem 
  we get again  $\gamma\leq 2\alpha-2+\frac{1}{p}-\frac{1}{q}$ and in the case $p=1$ or $q=\infty$    
\begin{align*}
  \gamma-(2\alpha-2+\frac{1}{p}-\frac{1}{q})
  &= d+1-2dm_{p,q} - \left(\frac{1}{p}-\frac{1}{q}\right) +\eps \\
  &\leq d+1-2dm_{p,q} - \frac{2d}{d+1}(2m_{p,q}+1) +\eps \\
  &= \frac{d^2+1}{d+1}-\frac{2d(d+3)}{d+1}m_{p,q}  +\eps \\
  &\leq   \frac{d^2+1}{d+1}-\frac{2d(d+3)}{d+1} \frac{d+\alpha}{2d}  +\eps \\
  &\leq \frac{-(3+\alpha)d+1-3\alpha}{d+1}    +\eps 
  < 0,
\end{align*}
which is all we had to show. \qed

\section{Proof of Theorem~\ref{thm:Slambdaalpha}} \label{sec:ProofSlambda}

  We have to prove the estimate
  $$
      \|\e{-\lambda\sqrt{|\cdot|^2-a^2}} m(|\cdot|)   \cF_d h(\cdot)\|_{L^s(A)} \les  \|h\|_{L^p(\R^d)} (1+\lambda)^{\frac{2}{s'}-\frac{2}{p}-\beta}.
  $$
  for $m\in L^\infty([a,b]),\lambda\geq 0$ and $\beta$ as in~\eqref{eq:def_beta}.
  For simplicity we assume $\mu_1=1,\mu_2=2$, i.e., $A=\{\xi\in\R^d: 1<|\xi|\leq 2\}$. We first present the
  bound given by the Hausdorff-Young inequality, so we assume  $d\in\N,1\geq \frac{1}{p_1}\geq
  \frac{1}{2},1\geq \frac{1}{s_1}\geq \frac{1}{p_1'}$. H\"older's inequality implies 
  \begin{align*}
    \|S_\lambda h\|_{L^{s_1}(A)}
    &\les  \|\cF_d h\|_{L^{p_1'}(A)}
    \|\e{-\lambda\sqrt{|\cdot|^2-1}}\|_{L^{\frac{s_1p_1'}{p_1'-s_1}}(A)}  \\
    &\les   \|h\|_{L^p(\R^d)}
    \left(  \int_1^2 \e{-\lambda\frac{s_1p_1'}{p_1'-s_1}\sqrt{r^2-1}}
    r^{d-1}\dr\right)^{\frac{p_1'-s_1}{s_1p_1'}}
    \\
    &\les  \|h\|_{L^p(\R^d)}
    \left(  \int_0^{\sqrt 3} \e{-\lambda\frac{s_1p_1'}{p_1'-s_1}\rho} \rho\drho
    \right)^{\frac{1}{s_1}-\frac{1}{p_1'}} \\
    &\les  \|h\|_{L^p(\R^d)}   (1+\lambda)^{-\frac{2}{s_1}+\frac{2}{p_1'}}. 
  \end{align*}
  This already gives the claim for $d=1$. So let us assume $d\geq 2$ from now on. We interpolate the previous estimate 
  with the following one for $1\geq \frac{1}{p_2}>\frac{1}{p_*(d)},1\geq \frac{1}{s_2}\geq
  \frac{d+1}{(d-1)p_2'}$. From Theorem~\ref{thm:RC} and Theorem~\ref{thm:Tao} we deduce the bound
  \begin{align*}
    \|S_\lambda h\|_{L^{s_2}(A)}
    &\les   \|\cF_d h\, \e{-\lambda\sqrt{|\cdot|^2-1}}\|_{L^{s_2}(A)} \\
    &\les  \left(\int_1^2  
    \e{-\lambda s_2\sqrt{r^2-1}}  \left(\int_{\Ss_r^{d-1}} |\cF_d h|^{s_2} \,\mathrm{d}\sigma_r\right) 
    \dr\right)^{\frac{1}{s_2}}  \\
    &\les   \left(\int_1^2  
    \e{-\lambda s_2\sqrt{r^2-1}}  r^{d-1-\frac{ds}{p'}} \|h\|_{L^p(\R^d)}^{s_2} 
    \dr\right)^{\frac{1}{s_2}}  \\
    &\les   \|h\|_{L^p(\R^d)} \left(\int_1^2  \e{-\lambda s_2\sqrt{r^2-1}}  \dr\right)^{\frac{1}{s_2}} 
    \\
    &\les   \|h\|_{L^p(\R^d)} (1+\lambda)^{-\frac{2}{s_2}}. 
  \end{align*}
  We infer from the Riesz-Thorin Theorem  
  \begin{equation}\label{eq:Slambdabound}
    \|S_\lambda h\|_{L^s(A)}  \les   \|h\|_{L^p(\R^d)}
    (1+\lambda)^{\frac{2}{s'}-\frac{2}{p_1}-\frac{2\theta}{p_1'}} 
  \end{equation}
  whenever  
  \begin{align*}
    \frac{1}{p}=\frac{1-\theta}{p_1}+\frac{\theta}{p_2},\qquad
    \frac{1}{s}=\frac{1-\theta}{s_1}+\frac{\theta}{s_2},\qquad 0\leq \theta\leq 1,\qquad\quad\\ 
    1\geq \frac{1}{p_1}\geq \frac{1}{2},\quad 1\geq \frac{1}{s_1}\geq
    \frac{1}{p_1'},\quad
    1\geq \frac{1}{p_2}> \frac{1}{p_*(d)},\quad 1\geq \frac{1}{s_2}\geq
    \frac{d+1}{(d-1)p_2'}.
  \end{align*}
  In order to get the asserted result we (subsequently) choose for sufficiently small $\tilde\eps>0$
  \begin{align*} %\label{eq:choices_interpolation}
    %\begin{aligned}
    &\theta=\min\left\{-1+\frac{d+1}{p}-\frac{d-1}{s'},1,\left(\frac{1}{p}-\frac{1}{2}\right)\left(\frac{1}{p_*(d)}-\frac{1-\tilde\eps}{2}\right)^{-1}
    \right\}, \\
    &\frac{\theta}{p_2} =
    \max\left\{\theta-\frac{d-1}{2}\left(\frac{1}{p}-\frac{1}{s'}\right),\frac{\theta}{p_*(d)}+\frac{\tilde\eps}{2}
    , \frac{1}{p}-1+\theta \right\},\qquad \frac{1-\theta}{p_1} =  \frac{1}{p}-\frac{\theta}{p_2}, \\
     &\max\left\{\frac{1}{s}-1+\theta,\frac{\theta(d+1)}{(d-1)p_2'}\right\} 
     \leq \frac{\theta}{s_2} \leq \min\left\{\theta,\frac{1}{s}-\frac{1-\theta}{p_1'}\right\}  
    ,\qquad 
    \frac{1-\theta}{s_1} = \frac{1}{s}-\frac{\theta}{s_2}.
    %\end{aligned}
  \end{align*}
  We briefly explain why this choice is admissible. The inequalities $1\geq
  \frac{1}{p_2}>\frac{1}{p_*(d)}$ and $1\geq \frac{1}{s_2}\geq \frac{d+1}{(d-1)p_2'}$ are immediate consequences of the definition of $p_2,s_2$.
  Moreover, $p_2\geq \frac{\theta}{\frac{1}{p}-1+\theta}$ implies $1\geq \frac{1}{p_1}$ and after some computations one finds that
  $\theta\leq\min\{-1+\frac{d+1}{p}-\frac{d-1}{s'},(\frac{1}{p}-\frac{1}{2})(\frac{1}{p_*(d)}-\frac{1}{2}+\frac{\tilde\eps}{2})^{-1}
  \}$ implies $\frac{1}{p_1}\geq \frac{1}{2}$. Finally, $s_2\leq \frac{\theta}{\frac{1}{s}-1+\theta}$ yields
  $1\geq \frac{1}{s_1}$ and $s_2\geq \frac{\theta}{\frac{1}{s}-\frac{1-\theta}{p_1'}}$ gives
  $\frac{1}{s_1}\geq \frac{1}{p_1'}$. With this choice we obtain for $\tilde\eps$ sufficiently small (in
  particular $\tilde\eps\leq \eps$)
  \begin{align*}
    & -\frac{2}{p_1}-\frac{2\theta}{p_1'} \\
    &=  - 2\theta-\frac{2(1-\theta)}{p_1} \\
    &=  -\frac{2}{p} -2\theta + \frac{2\theta}{p_2} \\
    &=  -\frac{2}{p} -2\theta + 2
    \max\left\{\theta-\frac{d-1}{2}\left(\frac{1}{p}-\frac{1}{s'}\right),\frac{\theta}{p_*(d)}+\frac{\tilde\eps}{2},
    \frac{1}{p}-1+\theta\right\} \\
    &=   -\frac{2}{p} -
    \min\left\{  \frac{d-1}{p}-\frac{d-1}{s'} ,
    \frac{2\theta}{p_*(d)'}- \tilde\eps,
    \frac{2}{p'}\right\}    \\
    &\leq  -\frac{2}{p} -
    \min\left\{ \frac{d-1}{p}-\frac{d-1}{s'},
    \frac{2(\frac{d+1}{p}-\frac{d-1}{s'}-1)}{p_*(d)'}-\eps,
    \frac{2}{p_*(d)'}-\eps,
    \frac{\frac{2}{p_*(d)'}(\frac{1}{p}-\frac{1}{2})}{\frac{1}{p_*(d)}-\frac{1}{2}} -\eps,   
    \frac{2}{p'}\right\}  \\
    &=  -\frac{2}{p} -
    \min\left\{ \frac{d-1}{p}-\frac{d-1}{s'},
    \frac{2(\frac{d+1}{p}-\frac{d-1}{s'}-1)}{p_*(d)'}-\eps,
    \frac{\frac{2}{p_*(d)'}(\frac{1}{p}-\frac{1}{2})}{\frac{1}{p_*(d)}-\frac{1}{2}} -\eps,   
    \frac{2}{p'}\right\}.   
  \end{align*}
  Here, the last equality comes from the fact that the third number inside the bracket of the second
  last line  lies between the fourth and the fifth number. Combining this with~\eqref{eq:Slambdabound} gives
  the desired bound.

\section{Proof of Proposition~\ref{prop:smallfreq}} \label{sec:smallfreq}

  In this section we prove Proposition~\ref{prop:smallfreq} dealing with the small frequency part 
  $w_\eps$ of the solution of the perturbed Helmholtz equation. In order to avoid heavy notation we
  carry out the estimates for $w=\lim_{\eps\searrow 0}w_{\eps}$ in detail and briefly discuss the
  necessary modifications afterwards. We recall from \eqref{def_wjWj} the formula
  \begin{align}\label{eq:formula_wj}
    \begin{aligned}
    w(x,y)&:=  \cF_{n-1}^{-1}\left(  \e{\i|y|\nu_{1}} (\ind_{|\cdot|\leq \mu_1} 
     m_1  g_+ +\ind_{|\cdot|\leq \mu_1}  m_2  g_-)\right)(x) \\
    &\qquad + \cF_{n-1}^{-1}\left( \e{\i|y|\nu_{2}} (\ind_{|\cdot|\leq \mu_1} m_3  g_+
    +  \ind_{|\cdot|\leq \mu_2} m_4 g_-)   \right)(x)
    \end{aligned}
  \end{align}
  where $m_1,\ldots,m_4$ were introduced in~\eqref{eq:def_mj}. We recall $\nu_j(\xi)=\sqrt{\mu_j^2-|\xi|^2}$
  for $|\xi|\leq \mu_j$ and $j=1,2$.
  %We will use $\nu_j(\xi)= \sqrt{\mu_j^2-|\xi|^2}$ for $|\xi|\leq \mu_j$. 
  We  have to prove the estimate
  $$ 
    \|w\|_{L^q(\R^n)}  + \sup_{0<|\eps|\leq 1} \|w_\eps\|_{L^q(\R^n)} \les \|f\|_{L^p(\R^n)}.
  $$ 
  under the assumptions  $\frac{1}{p}>\frac{1}{p_*(n)}$, $\frac{1}{q}< \frac{1}{q_*(n)}$ and
  $\frac{1}{p}-\frac{1}{q}\geq \frac{2}{n+1}$. 

\medskip

\noindent\textbf{Proof of Proposition~\ref{prop:smallfreq}:}
  For every fixed $x\in\R^{n-1},y\in\R$ we have  
    \begin{align*}
      w(x,y)
      &= (2\pi)^{\frac{1-n}{2}}   \int_{|\xi|\leq \mu_1}
      \e{\i(x\cdot\xi+|y|\nu_1(\xi))} (m_1(\xi )g_+(\xi)+m_2( \xi )g_-(\xi)) \dxi \\
      &\quad + (2\pi)^{\frac{1-n}{2}}   \int_{|\xi|\leq \mu_2}
      \e{\i(x\cdot\xi+|y|\nu_2(\xi))} (\ind_{|\xi|\leq \mu_1} m_3(\xi)g_+(\xi)+m_4( \xi )g_-(\xi)) \dxi
         \\
      &= (2\pi)^{\frac{1-n}{2}}  \int_{\Ss^{n-1}_{\mu_1}}
      \e{\i(x \cdot\xi+|y|\eta)}  \frac{m_1(\xi)g_+(\xi)+m_2(\xi)g_-(\xi)}{(1+|\nabla
      \nu_1(\xi)|^2)^{\frac{1}{2}}} 1_{(0,\infty)}(\eta)  \, \mathrm{d}\sigma_{\mu_1}(\xi,\eta)   \\
      &\quad + (2\pi)^{\frac{1-n}{2}}  \int_{\Ss^{n-1}_{\mu_2}}
      \e{\i(x \cdot\xi+|y|\eta)}  \frac{\ind_{|\xi|\leq \mu_1} m_3(\xi)g_+(\xi)+m_4(\xi)g_-(\xi)}{
      (1+|\nabla \nu_2(\xi)|^2)^{\frac{1}{2}}} 1_{(0,\infty)}(\eta)  \, \mathrm{d}\sigma_{\mu_2}(\xi,\eta).
    \end{align*}
    Next we use Theorem~\ref{thm:RC} in the case $n=2$ and Tao's Fourier Restriction Theorem
    (Theorem~\ref{thm:Tao}) in the case $n\geq 3$. In both cases, $s:= \left(\frac{n-1}{n+1} \, q\right)'$.
    Using the estimates $|m_1|\les |\nu_1|^{-1},|m_2|+|m_3|\les 1,|m_4|\les |\nu_2|^{-1}$, which follow
    from~\eqref{eq:mj_estimates},  we obtain
   \begin{align*}
     \|w\|_{L^q(\R^n)} 
      &\les   \|(m_1 g_+ +m_2 g_-) (1+|\nabla
      \nu_1|^2)^{-\frac{1}{2}}\|_{L^s(\Ss^{n-1}_{\mu_1})} \\
      &\quad + \|(\ind_{|\xi|\leq \mu_1} m_3 g_+ +m_4g_-) (1+|\nabla
      \nu_2|^2)^{-\frac{1}{2}}\|_{L^s(\Ss^{n-1}_{\mu_2})}  \\
      &\les   \| (|\nu_1|^{-1}|g_+|+ |g_-|) (1+|\nabla   \nu_1|^2)^{-\frac{1}{2}}\|_{L^s(\Ss^{n-1}_{\mu_1})}
      \\
      &\quad + \| (\ind_{|\xi|\leq \mu_1} |g_+|+ |\nu_2|^{-1} |g_-|) 
      (1+|\nabla \nu_2|^2)^{-\frac{1}{2}}\|_{L^s(\Ss^{n-1}_{\mu_2})}  \\
      &\stackrel{\eqref{eq:estimate_nu}}\les   \| |g_+|+ |\nu_1| |g_-| \|_{L^s(\Ss^{n-1}_{\mu_1})} + 
      \| \ind_{|\xi|\leq \mu_1} |\nu_2| |g_+|+ |g_-|\|_{L^s(\Ss^{n-1}_{\mu_2})}  \\
      &\les   \| |g_+|+ |g_-| \|_{L^s(\Ss^{n-1}_{\mu_1})} + 
      \| \ind_{|\xi|\leq \mu_1}  |g_+|+ |g_-|\|_{L^s(\Ss^{n-1}_{\mu_2})}  \\
      &\les   \|g_+\|_{L^s(\Ss^{n-1}_{\mu_1})} +  \|g_-\|_{L^s(\Ss^{n-1}_{\mu_2})}  \\
      &\les \|\cF_n^+f\|_{L^s(\Ss^{n-1}_{\mu_1})} + \| \cF_n^-f\|_{L^s(\Ss^{n-1}_{\mu_2})} .
  \end{align*}
  Since $\frac{1}{p}-\frac{1}{q}\geq \frac{2}{n+1}$ implies $s'\geq \big(\frac{n-1}{n+1} \,
  p'\big)'$ and $p'>\frac{2(n+2)}{n}$, Theorem~\ref{thm:Tao} applies and we get
  $$
     \|w\|_{L^q(\R^n)} 
    \les \|\cF_n^+f\|_{L^s(\Ss^{n-1}_{\mu_1})} + \| \cF_n^-f\|_{L^s(\Ss^{n-1}_{\mu_2})}
    \les \|f\|_{L^p(\R^n)},
  $$
  which is all we had to show. Here we used $\cF_n^\pm f =\cF_n(f_\pm)$ where
  $f_\pm(x,y)=f(x,y)1_{(0,\infty)}(\pm y)$.
  
   \medskip
   
   Now we indicate the necessary modifications to get the corresponding uniform estimates for
   $w_\eps$ with respect to $\eps\in (0,1]$. Here, each $\nu_j$ in~\eqref{eq:formula_wj} is replaced by 
   $\nu_{j,\eps}$ where $\nu_{j,\eps}(\xi)^2 =  \mu_j^2-|\xi|^2+\i\eps$ and $\Imag(\nu_{j,\eps}(\xi))>0$.
   In this case we obtain the same estimates as above because the sets  
   $\{(\xi,\Real(\nu_{j,\eps}(\xi))):|\xi|\leq  \mu_j\}%=\Ss^{n-1}_{\mu_j,\eps}\cap\{\eta\geq \sqrt{\eps/2}\}
   $ are regular hypersurfaces with the property that the Gaussian curvature has a positive lower bound 
   independent of $\eps$. For such surfaces Tao's result remains true and we may thus argue as above. Notice
   that the positive imaginary part of $\nu_{1,\eps},\nu_{2,\eps}$ lead to  damping factors 
   $\e{-|y|\Imag(\nu_{j,\eps}(\xi))}$ in the integrals that may be estimated from above by one.
\qed

\section{Proof of Proposition~\ref{prop:interferencefreq}} \label{sec:interferencefreq}

  In this section we bound (the first part of) the intermediate frequency terms given by 
\begin{align*}
  \mathfrak w(x,y) 
  &= \cF_{n-1}^{-1}\left( \e{\i y\nu_1} 1_A m_2 g_- + \e{\i y \nu_2} 1_A m_3 g_+\right)(x) \\
  &\stackrel{\eqref{eq:def_mj}}= 1_{y>0} \cF_{n-1}^{-1}\left( \e{\i y\nu_1} 1_A m^*
  \cF_n^-f(\cdot,-\nu_2(\cdot) \right)(x) \\
    &\qquad + 1_{y<0} \cF_{n-1}^{-1}\left( \e{\i y \nu_2} 1_A m^* \cF_n^+f(\cdot,-\nu_1(\cdot))\right)(x) 
\end{align*}
where $m^*(\xi) :=   \i\sqrt{2\pi}(\nu_1(\xi)+\nu_2(\xi))^{-1}$.
  Here, $A=\{\xi\in\R^{n-1}: \mu_1<|\xi|\leq \mu_2\}$ so that $\xi\in A$ implies   
  $\nu_1(\xi)= \i(|\xi|^2-\mu_1)^{1/2}$ and $\nu_2(\xi)=(\mu_2^2-|\xi|^2)^{-1/2}$, see~\eqref{eq:nu}.
For a bounded complex-valued function $\mathfrak m\in L^\infty(A)$   
we define the linear operators
\begin{align*}
  Q_{\mathfrak m} h(x,y)
  &:= 1_{y<0}  \cF_{n-1}^{-1}\left(
     \e{\i y\nu_{2}} \ind_{A}  \mathfrak m \cF_n^+h(\cdot,-\nu_1(\cdot))\right)(x)    
\end{align*}
that we will prove to be bounded from $L^p(\R^n)$ to $L^q(\R^n)$ under the given assumptions on $p,q$. 
Its adjoint is then bounded from $L^{q'}(\R^n)$ to $L^{p'}(\R^n)$ and hence as well for all $p,q$ according
to the assumptions. It is given by the formula   
\begin{align*}
  Q_{\mathfrak m}^*h(x,y)
  &:= 1_{y>0} \cF_{n-1}^{-1}\left(
     \e{\i y\nu_1} \ind_{A}  \ov{\mathfrak m} \cF_n^- h(\cdot,\nu_2(\cdot))\right)(x) 
\end{align*}
because we have for all $h_1,h_2\in\mathcal S(\R^n)$
\begin{align*}
  & \int_\R \int_{\R^{n-1}} Q_{\mathfrak m} h_1(x,y) \ov{h_2(x,y)}  \,\mathrm{d}x \, \mathrm{d}y \\ 
  &= \int_{-\infty}^0 \int_{\R^{n-1}}  
     \e{\i y \nu_2(\xi) } \ind_A(\xi) \mathfrak m(\xi)  
     \cF_n^+ h_1(\xi,-\nu_1(\xi)) \overline{\cF_{n-1}[h_2(\cdot,y)](\xi)} \,\mathrm{d}\xi \,
     \mathrm{d}y\\
 &=   \int_{\R^{n-1}}     
	 \ind_A(\xi) \mathfrak m(\xi)   \cF_n^+h_1(\xi,-\nu_1(\xi)) 
	 \overline{\int_{-\infty}^0 
	 \e{-\i y \nu_2(\xi)} 
	 \mathcal{F}_{n-1}[h_2(\cdot, y)](\xi) 
	 \, \mathrm{d}y} \,\mathrm{d}\xi\\
 &=   \int_{\R^{n-1}}       
	 \ind_A(\xi) \mathfrak m(\xi)  \cF_n^+h_1(\xi,-\nu_1(\xi)) 
	  \cdot (2\pi)^\frac{1}{2}
	 \overline{  \cF_n^- h_2(\xi,\nu_2(\xi)) 
	 } \,\mathrm{d}\xi\\	 
 &=    \int_{\R^{n-1}} \left(\int_0^\infty       
	 \ind_A(\xi)  \cF_{n-1}[h_1(\cdot,y)](\xi) \e{\i\nu_1(\xi)}  \dy 
	 \right)  
	  \mathfrak m(\xi) 
	 \overline{  \cF_n^- h_2(\xi,\nu_2(\xi)) 
	 } \,\mathrm{d}\xi\\
&=   \int_0^\infty \int_{\R^{n-1}}        
	  \cF_{n-1}[h_1(\cdot,y)](\xi) \:
	 \overline{  
	   \e{\i\nu_1(\xi)} \ind_A(\xi)  \ov{\mathfrak m(\xi)}
	 \cF_n^- h_2(\xi,\nu_2(\xi)) 
	      } \,\mathrm{d}\xi \dy\\
&=   \int_0^\infty \int_{\R^{n-1}}        
	  h_1(x,y) \:
	 \overline{\cF_{n-1}^{-1}\left(  
	   \e{\i\nu_1} \ind_A   \ov{\mathfrak m}
	 \cF_n^- h_2(\cdot,\nu_2(\cdot)) 
	      \right)(x)} \,\mathrm{d}x\dy\\
	&=  \int_\R \int_{\R^{n-1}} h_1(x,y) \ov{Q_{\ov{\mathfrak m}}^*h_2 (x,y)}  
	 \,\mathrm{d}x   \,\mathrm{d}y.  
\end{align*}
The following result tells us that it is sufficient to find $L^p-L^q$-bounds for $Q_{\mathfrak m}$.

\begin{prop} \label{prop:interference_representation}
  For $x\in\R^{n-1},y>0$ we have
  $$
    \mathfrak w(x,y)  =  (Q_{\ov{m^*}}^* f)(x,y)   + (Q_{m^*}f)(x,y).
  $$
\end{prop}

Our bounds for $Q_m$ rely on Theorem~\ref{thm:Slambdaalpha}.

\begin{lem} \label{lem:Q_bounds}
  Let $n\in\N,n\geq 2$ and $\mathfrak m\in L^\infty(A)$. Then the linear operator $Q_{\mathfrak
  m}:L^p(\R^n)\to L^q(\R^n)$ is bounded whenever 
  $\frac{1}{p}>\frac{1}{p_*(n)},\frac{1}{q}<\frac{1}{q_*(n)},\frac{1}{p}-\frac{1}{q}\geq
  \frac{2}{n+1}$. In particular, this holds for all $(p,q)\in\tilde {\mathcal D}$. 
  If the Restriction Conjecture is true then it is bounded 
  whenever  $\frac{1}{p}>\frac{n+1}{2n},\frac{1}{q}<\frac{n-1}{2n},\frac{1}{p}-\frac{1}{q}\geq
  \frac{2}{n+1}$ and hence for all $(p,q)\in \mathcal D$. 
\end{lem} 
\begin{proof}
  We have  
\begin{align*}
 	Q_{\mathfrak m}h(x,y)
 	&=  \cF_{n-1}^{-1}\left(
     \e{\i y\nu_{2}} \ind_{A}  \mathfrak m \cF_n^+h(\cdot,-\nu_1(\cdot))\right)(x) \cdot 1_{y<0}\\
 	& =    (2\pi)^{\frac{1-n}{2}} \int_{|\xi| \leq \mu_2} 
 	\e{\i (x \cdot \xi - y \nu_2(\xi))} \ind_{A}(\xi) \mathfrak m(\xi) \cF_n^+h(\xi,-\nu_1(\xi)) 
 	\: \mathrm{d}\xi \cdot 1_{y<0}\\
 	& =    (2\pi)^{\frac{1-n}{2}} \int_{\mathbb{S}^{n-1}_{\mu_2}} 
 	\e{\i (x \cdot \xi - y \eta)} \ind_{A}(\xi) \ind_{\eta>0} \mathfrak{m}(\xi) 
 	\frac{\cF_n^+h(\xi,-\nu_1(\xi))}{(1 + |\nabla \nu_2(\xi)|^2)^{-\frac{1}{2}}}
 	\: 	\mathrm{d}\sigma_{\mu_2}(\xi,\eta) \cdot 1_{y<0}
%  	\\
%  	& =   \sqrt{2\pi} \: \cF_n^{-1}\left[ 
%  		\frac{\ind_{A} m_3(\cdot) \:  g_+}{\sqrt{1 + |\nabla \nu_2|^2}} \mathrm{d}\sigma_{\mu_2}
%  	\right](x, -y). 
\end{align*}
  In the case $n=2,n\geq 3$ we use Theorem~\ref{thm:RC}, Theorem~\ref{thm:Tao}, respectively. Due to 
  $\frac{1}{q}<\frac{1}{q_*(n)}$ and $s:=\left( \frac{n-1}{n+1} q \right)'$  we get the bound   
\begin{align*}
  \|Q_{\mathfrak m}h\|_{L^q(\R^n)}
 &\les
 	 \norm{\ind_{A}  \mathfrak{m}  \cF_n^+h(\cdot,-\nu_1(\cdot)) 
 	(1 + |\nabla \nu_2(\cdot)|^2)^{-\frac{1}{2}}
 	}_{L^s(\mathbb{S}^{n-1}_{\mu_2})}
 	\\
   &\les \|\mathfrak{m}\|_\infty
 	 \norm{ \ind_{A}  \cF_n^+h(\cdot,-\nu_1(\cdot)) }_{L^s(\mathbb{S}^{n-1}_{\mu_2})}\\	
   &\les \|\mathfrak{m}\|_\infty   
 		\left( \int_{A}
 		 \left| \int_0^\infty \cF_{n-1}[f (\cdot , z)](\xi) \e{-z \sqrt{|\xi|^2 - \mu_1^2}} \, \mathrm{d}z
 		 \right|^s \, \mathrm{d}\xi
 		\right)^\frac{1}{s}.
 		\\
 		\intertext{Minkowski's inequality and Theorem~\ref{thm:Slambdaalpha} imply for $\beta$ and $d=n-1$ 
 		as in~\eqref{eq:def_beta}} 
 		\|Q_{\mathfrak m}h\|_{L^q(\R^n)}	
 	&\les \|\mathfrak{m}\|_\infty 
 		 \int_0^\infty 
 		 \left( \int_{A}
 		 \left|\cF_{n-1}[h (\cdot , z)](\xi) \e{-z \sqrt{|\xi|^2 - \mu_1^2}} 
 		\right|^s 
 		\, \mathrm{d}\xi
 		\right)^\frac{1}{s} 
 		 \, \mathrm{d}z 
 		 \\
 		 &\les \|\mathfrak{m}\|_\infty
 		 \int_0^\infty \|h(\cdot,z)\|_{L^p(\R^{n-1})} (1+z)^{\frac{2}{s'}-\frac{2}{p}-\beta}  \, \mathrm{d}z \\
 		 & \les \|\mathfrak{m}\|_\infty \norm{h}_{L^p(\R^n)}
\end{align*}
provided that $(\frac{2}{s'}-\frac{2}{p}-\beta)p'<-1$ or equivalently $\beta+\frac{3}{p}-\frac{2}{s'}-1>0$.
To prove the main statement of the Lemma, it remains to check this condition for all $(p,q)\in
\tilde{\mathcal D}$. Indeed, in the case $n=2$, where $d=n-1=1$, this follows from $\beta=0$ and  
$$
  \frac{3}{p}-\frac{2}{s'}-1
  = 3\left(\frac{1}{p}-\frac{1}{q}\right) - \frac{3}{q}-1
  > 2  - \frac{3}{4} -1
  >0.
$$ 
In the case $d=n-1\geq 2$ this is a consequence of the definition of
$\beta$ from~\eqref{eq:def_beta}. We have
\begin{align*}
  \left(\frac{n-2}{p}-\frac{n-2}{s'}\right)+\frac{3}{p}-\frac{2}{s'}-1
  &= \frac{n+1}{p}-\frac{n}{s'}-1 
  \;=\; \frac{n+1}{p}-\frac{n(n+1)}{(n-1)q}-1 \\
  &= (n+1)\left(\frac{1}{p}-\frac{1}{q}\right) - \frac{n+1}{(n-1)q}-1 \\
  &>  2 - \frac{n+1}{n-1}\cdot \frac{n}{2(n+2)} - 1\\% \b{1-\frac{n+1}{2n}}\\
  &= \frac{n^2+n-4}{2(n-1)(n+2)} \\
  & > 0, \\ %\quad \b{\frac{n-1}{2n}} \\
  \frac{2}{p'}+\frac{3}{p}-\frac{2}{s'}-1
  &=  \frac{1}{p}-\frac{2(n+1)}{(n-1)q} + 1 \\
  &= \left(\frac{1}{p}-\frac{1}{q}\right)  - \frac{n+3}{(n-1)q} +1 \\
  &> \frac{2}{n+1} - \frac{n+3}{n-1}\cdot \frac{n}{2(n+2)} +1  \\%
  % \b{\frac{n+3}{n+1}-\frac{n+3}{2n}}\\
  &= \frac{(n+3)(n^2+n-4)}{2(n+1)(n-1)(n+2)} \\
  &> 0,  \\ % \b{\frac{n-1}{2n(n+1)}} \\
  \left(\frac{\frac{2}{(p_*(n-1))'}(\frac{1}{p}-\frac{1}{2})}{\frac{1}{p_*(n-1)}-\frac{1}{2}}\right)
  +\frac{3}{p}-\frac{2}{s'}-1
  &= (n-1)\left(\frac{1}{p}-\frac{1}{2}\right)
  +\frac{3}{p}-\frac{2}{s'}-1 \\ 
  &= \frac{n+2}{p}-\frac{2(n+1)}{(n-1)q}  -\frac{n+1}{2}  \\
  &=  \frac{n^2-n-4}{(n-1)p}  + \frac{2(n+1)}{n-1}\left(\frac{1}{p}-\frac{1}{q}\right)  -\frac{n+1}{2}  \\
  &>  \frac{(n^2-n-4)(n+4)}{2(n-1)(n+2)}  + \frac{4}{n-1}   -\frac{n+1}{2}  \\
  &= \frac{n^2+n+2}{2(n+2)(n-1)}
  \\ 
  & > 0, \\
  \left(\frac{2(\frac{n}{p}-\frac{n-2}{s'}-1)}{p_*(n-1)'}\right)
  +\frac{3}{p}-\frac{2}{s'}-1 
  &= \left(\frac{(n-1)n}{(n+1)p}-\frac{(n-1)(n-2)}{(n+1)s'}-\frac{n-1}{n+1} \right)
  +\frac{3}{p}-\frac{2}{s'}-1 \\
  &=  \frac{n^2+2n+3}{(n+1)p} - \frac{n^2-n+4}{(n+1)s'} - \frac{2n}{n+1} \\
  &=  \frac{n^2+2n+3}{n+1}\left(\frac{1}{p}-\frac{1}{q}\right) + \frac{n^2-2n-7}{(n+1)(n-1)q} -
  \frac{2n}{n+1} \\
  &>  \frac{2(n^2+2n+3)}{(n+1)^2}  + 1_{n=3} \frac{(n^2-2n-7)n}{2(n+1)(n-1)(n+2)} -
  \frac{2n}{n+1} \\
  &=  \frac{2(n+3)}{(n+1)^2}  + 1_{n=3} \frac{(n^2-2n-7)n}{2(n+1)(n-1)(n+2)}    
  \\
  & > 0. 
\end{align*}
Hence, the condition $(\frac{2}{s'}-\frac{2}{p}-\beta)p'<-1$ holds. For the extra claim regarding  
the Restriction Conjecture, is suffices to prove the above estimates (for $n\geq 3$) where
$p_*(n-1)$ is replaced by $\frac{2(n-1)}{n}$ and the estimates
$\frac{1}{p}>\frac{n+4}{2(n+2)},\frac{1}{q}<\frac{n}{2(n+2)}$ are replaced by $\frac{1}{p}>\frac{n-1}{2n},
\frac{1}{q}<\frac{n-1}{2n}$. This can be done as above and one obtains again
$\beta+\frac{3}{p}-\frac{2}{s'}-1>0$ and the proof is finished.
\end{proof}

\medskip

\noindent\textbf{Proof of Proposition~\ref{prop:interferencefreq}:} This is a consequence of
Proposition~\ref{prop:interference_representation} and Lemma~\ref{lem:Q_bounds} because 
$$
  \|\mathfrak w\|_{L^q(\R^n)} 
  \les \|Q_{\ov{m^*}}^* f\|_{L^q(\R^n)}    + \| Q_{m^*}f\|_{L^q(\R^n)}
  \les \|f\|_{L^p(\R^n)}
$$
provided that 
$\frac{1}{p}>\frac{1}{p_*(n)},\frac{1}{q}<\frac{1}{q_*(n)},\frac{1}{p}-\frac{1}{q}\geq \frac{2}{n+1}$ holds.
If the Restriction Conjecture holds, $p_*(n)$ and $q_*(n)$ can be replaced by $\frac{2n}{n+1}$ and
$\frac{2n}{n-1}$, respectively.
\qed

% 
% 
% $$
%   \alpha>\frac{1}{2}-\frac{1}{2p}-\frac{1}{s}+\frac{1}{r},
%   \quad
%   1\leq p<\frac{2(d+2)}{d+2+2\alpha},\quad
%   \frac{1}{r}\geq \frac{1}{p'}+\frac{((\alpha+1)p-2)_+}{(d-1)p},\quad r\geq s
% $$
% We choose  and  
% $$
%   \frac{1}{r}
%   =\frac{1}{p'}+\frac{(\alpha+1)p-2}{(d-1)p}
%   = 1- \frac{1}{p}+ \frac{\alpha+1}{d-1} - \frac{2}{(d-1)p}
%   =  \frac{d+\alpha}{d-1} -  \frac{d+1}{(d-1)p}
% $$  

% \begin{align*}
%   \alpha>\frac{1}{2}-\frac{1}{2p}-\frac{1}{s}+\frac{1}{r} \\
%   \alpha>\frac{1}{2}-\frac{1}{2p}- (1-\frac{n+1}{(n-1)q})+ \frac{d+\alpha}{d-1} -  \frac{d+1}{(d-1)p} \\
%   \frac{d-2}{d-1}\alpha>  \frac{n+1}{(n-1)q}+ \frac{d+1}{2(d-1)} - 
%   \frac{3d+1}{2(d-1)p} \\
%   \frac{n-3}{n-2}\alpha>  \frac{n+1}{(n-1)q}+ \frac{n}{2(n-2)} - 
%   \frac{3n-2}{2(n-2)p} \\
%   \frac{n-3}{n-2}>  \frac{n+1}{(n-1)q}+ \frac{n}{2(n-2)} - 
%   \frac{3n-2}{2(n-2)p},\quad 
%   \frac{n-3}{n-2} (\frac{n+1}{p}-\frac{n+1}{2})>  \frac{n+1}{(n-1)q}+ \frac{n}{2(n-2)} - 
%   \frac{3n-2}{2(n-2)p} \\
%   \frac{n-6}{2(n-2)} >  \frac{n+1}{(n-1)q}  - \frac{3n-2}{2(n-2)p},\quad
%   \frac{(n-3)(n+1)}{(n-2)p}  -\frac{(n-3)(n+1)}{2(n-2)}>  \frac{n+1}{(n-1)q}+ \frac{n}{2(n-2)} - 
%   \frac{3n-2}{2(n-2)p} \\
%   \frac{n-6}{2(n-2)} >  \frac{n+1}{(n-1)q}  - \frac{3n-2}{2(n-2)p},\quad
%   \frac{2n^2-n-8}{2(n-2)p}  >  \frac{n+1}{(n-1)q}+ \frac{n^2-n-3}{2(n-2)}   
% \end{align*}
%

\section{Proof of Proposition~\ref{prop:intermediatefreq}} \label{sec:intermediatefreq}
 
 We now prove our estimates for those intermediate frequencies collected in $\mathfrak{W}_\eps$. As
 above, we avoid the technicalities for $\eps>0$ by considering only the most singular limit term 
  \begin{align*} 
    \begin{aligned}
   \mathfrak{W}(x,y)
    = \lim_{\eps\searrow 0} \mathfrak W_{\eps}(x,y) 
    &=\cF_{n-1}^{-1}\left(  \e{\i|y|\nu_{1}} \ind_{\mu_1<|\cdot|\leq \mu_1+\mu_2} 
    m_1  g_+ 
    + \e{\i|y|\nu_{1}} \ind_{\mu_2<|\cdot|\leq \mu_1+\mu_2}  m_2  g_-\right)(x) \\
    &\;+ \cF_{n-1}^{-1}\left( 
     \e{\i|y|\nu_{2}}\ind_{\mu_2<|\cdot|\leq \mu_1+\mu_2}    m_3  g_+
    + \e{\i|y|\nu_{2}}\ind_{\mu_2<|\cdot|\leq \mu_1+\mu_2} m_4 g_-   \right)(x).
    \end{aligned}
  \end{align*}
  We use $\nu_j(\xi)= \i\sqrt{|\xi|^2-\mu_j^2}$ for $|\xi|>\mu_j$. To prove
  Proposition~\ref{prop:intermediatefreq} wehave to show that for all $n\in\N,n\geq 2$ and all $p,q$ such
  that $\frac{1}{p}>\frac{n+1}{2n}$, $\frac{1}{q}<\frac{n-1}{2n}$ and $\frac{1}{p}-\frac{1}{q}\geq
  \frac{2}{n+1}$ the following estimate holds 
  $$ 
    \|\mathfrak{W}\|_{L^q(\R^n)} + \sup_{0<|\eps|\leq 1}  \|\mathfrak{W}_{\eps}\|_{L^q(\R^n)} \les
    \|f\|_{L^p(\R^n)}.
  $$

\noindent\textbf{Proof of Proposition~\ref{prop:intermediatefreq}:}
  We introduce the annuli $A_j:=\{\xi\in\R^{n-1}: \mu_j\leq |\xi|\leq \mu_1+\mu_2\}$ for $j=1,2$. Then we have
  for fixed $x\in\R^n,y\in\R$ 
\begin{align*}
   \mathfrak{W}(x,y)
   &=  \cF_{n-1}^{-1}\left( 1_{A_1}   
    \e{\i |y|\nu_1} m_1  \cF_n^+f(\cdot,-\nu_1(\cdot))  \right)(x) \\
    &+ \cF_{n-1}^{-1}\left(  1_{A_2}   \e{\i|y|\nu_{1}} m_2     
     \cF_n^-f(\cdot,\nu_2(\cdot)) \right)(x) \\  
    &+      \cF_{n-1}^{-1}\left( 1_{A_2}  \e{\i|y|\nu_{2}}m_3 
     \cF_n^+f(\cdot,-\nu_1(\cdot))\right)(x) \\
     & +  \cF_{n-1}^{-1}\left(
    1_{A_2} \e{\i|y|\nu_{2}}     
    m_4 \cF_n^-f(\cdot,\nu_2(\cdot))\right)(x) \\
   &=  \int_0^\infty \cF_{n-1}^{-1}\left( 1_{A_1}(\xi)
    \e{-(|y|+z)\sqrt{|\xi|^2-\mu_1^2}} m_1(\xi) \cF_{n-1}[f(\cdot,z)](\xi)
     \right)(x) \dz \\
   &+ \int_{-\infty}^0   \cF_{n-1}^{-1}\left( 1_{A_2}(\xi)
    \e{-|y|\sqrt{|\xi|^2-\mu_1^2}+z\sqrt{|\xi|^2-\mu_2^2}} m_2(\xi) \cF_{n-1}[f(\cdot,z)](\xi) \right)(x) \dz
    \\
   &+ \int_0^\infty  \cF_{n-1}^{-1} \left(  1_{A_2}(\xi)
   \e{-|y|\sqrt{|\xi|^2-\mu_2^2}-z\sqrt{|\xi|^2-\mu_1^2}} m_3(\xi)
   \cF_{n-1}[f(\cdot,z)](\xi) \right)(x)\dz \\
    &+ \int_{-\infty}^0   \cF_{n-1}^{-1}\left( 1_{A_2}(\xi)
    \e{-|y|\sqrt{|\xi|^2-\mu_2^2}+z\sqrt{|\xi|^2-\mu_2^2}} m_4(\xi) \cF_{n-1}[f(\cdot,z)](\xi) \right)(x)
    \dz\\
    &=  \int_0^\infty \cF_{n-1}^{-1}\left( 1_{A_1}(\xi)
    \e{-(|y|+|z|)\sqrt{|\xi|^2-\mu_1^2}} (|\xi|^2-\mu_1^2)^{-\frac{1}{2}} \tilde m_1(|\xi|)
    \cF_{n-1}[f(\cdot,z)](\xi) \right)(x) \dz \\
    &+ \int_{-\infty}^0 \cF_{n-1}^{-1}\left( 1_{A_2}(\xi)
    \e{-(|y|+|z|)\sqrt{|\xi|^2-\mu_2^2}} \tilde  m_2(|\xi|)
    \cF_{n-1}[f(\cdot,z)](\xi) \right)(x) \dz \\
    &+\int_0^\infty \cF_{n-1}^{-1}\left( 1_{A_2}(\xi)
    \e{-(|y|+|z|)\sqrt{|\xi|^2-\mu_2^2}}  \tilde m_3(|\xi|)
    \cF_{n-1}[f(\cdot,z)](\xi) \right)(x) \dz\\
    &+ \int_{-\infty}^0 \cF_{n-1}^{-1}\left( 1_{A_2}(\xi)
    \e{-(|y|+|z|)\sqrt{|\xi|^2-\mu_2^2}} (|\xi|^2-\mu_2^2)^{-\frac{1}{2}}\tilde  m_4(|\xi|)
    \cF_{n-1}[f(\cdot,z)](\xi) \right)(x) \dz.
\end{align*}
Here, the functions $\tilde m_1,\ldots,\tilde m_4$ are defined by
\begin{align*}
  \tilde m_1(|\xi|)
  &= m_1(\xi)(|\xi|^2-\mu_1^2)^{\frac{1}{2}}  
   = \frac{ \sqrt{\pi/2}}{\nu_1(\xi)+\nu_2(\xi)}\cdot \left(\sign(y)\nu_1(\xi)-\nu_2(\xi)\right), \\
  \tilde m_4(|\xi|)
  &= m_4(\xi) (|\xi|^2-\mu_2^2)^{\frac{1}{2}} 
  = \frac{\sqrt{\pi/2}}{\nu_1(\xi)+\nu_2(\xi)}\cdot
  \left(-\sign(y)\nu_2(\xi)- \nu_1(\xi)\right)
\end{align*}
and
\begin{align*} 
  \tilde m_2(|\xi|)
  &= m_2(\xi)  \e{|y|(\sqrt{|\xi|^2-\mu_2^2}-\sqrt{|\xi|^2-\mu_1^2})} 
  = \frac{\i\sqrt{\pi/2}(1+\sign(y))}{\nu_1(\xi)+\nu_2(\xi)}\cdot 
  \e{|y|(\sqrt{|\xi|^2-\mu_2^2}-\sqrt{|\xi|^2-\mu_1^2})},  \\
  \tilde m_3(|\xi|)
  &= m_3(\xi) \e{|z|(\sqrt{|\xi|^2-\mu_2^2}-\sqrt{|\xi|^2-\mu_1^2})}  
   = \frac{\i\sqrt{\pi/2}(1-\sign(y))}{\nu_1(\xi)+\nu_2(\xi)}\cdot  
   \e{|z|(\sqrt{|\xi|^2-\mu_2^2}-\sqrt{|\xi|^2-\mu_1^2})}.  
\end{align*}
% \begin{align*}
%   m_1(|\xi|)
%   &= \frac{\i\sqrt{\pi/2}}{\nu_1(\xi)+\nu_2(\xi)}\cdot \left(\sign(y)-\frac{\nu_2(\xi)}{\nu_1(\xi)}\right) \\
%   m_2(|\xi|) 
%   &= \frac{\i\sqrt{\pi/2}}{\nu_1(\xi)+\nu_2(\xi)}\cdot \left(1+\sign(y)\right)\\
%   m_3(|\xi|)
%   &= \frac{\i\sqrt{\pi/2}}{\nu_1(\xi)+\nu_2(\xi)}\cdot \left(1-\sign(y)\right)\\
%   m_4(|\xi|)
%   &= \frac{\i\sqrt{\pi/2}}{\nu_1(\xi)+\nu_2(\xi)}\cdot \left(-\sign(y)-\frac{\nu_1(\xi)}{\nu_2(\xi)}\right) \\
% \end{align*}
Notice that $\tilde m_1,\ldots,\tilde m_4$ are indeed radially symmetric because so are $\nu_1,\nu_2$. 
From $|\nu_1(\xi)+\nu_2(\xi)|\gtrsim 1$ and $\mu_1<\mu_2$ we infer that all four terms are 
bounded independently of $y,z$.  So we apply 
Theorem~\ref{thm:Tlambdaalpha_d=1} $(n=2)$ resp. Theorem~\ref{thm:Tlambdaalpha} $(n\geq 3)$ to bound these
integrals for any fixed $y,z\in\R$. The assumptions of these theorems are satisfied  
because our assumption in Proposition~\ref{prop:intermediatefreq} implies for $d=n-1$ and
$\alpha\in\{0,\frac{1}{2}\}$ 
$$
  \frac{1}{p}>\frac{n+1}{2n} \geq \frac{1}{2}+\frac{\alpha}{2d},\qquad
  \frac{1}{q}<\frac{n-1}{2n} \leq \frac{1}{2}-\frac{\alpha}{2d},\qquad 
  \frac{1}{p}-\frac{1}{q}\geq \frac{2}{n+1}\geq \frac{2\alpha}{d+1}. 
$$
From the above-mentioned Theorems we get
\begin{align*}
  \|\mathfrak{W}(\cdot,y)\|_{L^q(\R^{n-1})}    
  &\les\int_\R  (1+|y|+|z|)^{\gamma} \|f(\cdot,z)\|_{L^p(\R^{n-1})} \:\mathrm{d}z
\end{align*}
where $\gamma\leq  -1+\frac{1}{p}-\frac{1}{q}$  
and $\gamma<-1+\frac{1}{p}-\frac{1}{q}$ for  $p=1$ or $q=\infty$. In the latter case
the classical version of Young's convolution inequality applies and gives
\begin{align*}
  \|\mathfrak{W}\|_{L^q(\R^n)}
  &\les \|(1+|\cdot|)^\gamma \ast \|f(\cdot,z)\|_{L^p(\R^{n-1})}\|_{L^q(\R)} \\
  &\les \|(1+|\cdot|)^\gamma\|_{L^{\frac{pq}{pq-q+p}}(\R)} 
  \left(\int_\R \|f(\cdot,z)\|_{L^p(\R^{n-1})}^p\:\mathrm{d}z\right)^{\frac{1}{p}}   \\
  &\les \|f\|_{L^p(\R^n)}. 
  \intertext{ 
In the former case Young's convolution inequality in weak Lebesgue spaces \cite[Theorem~1.4.25]{Grafakos} is
applicable and yields} 
  \|\mathfrak{W}\|_{L^q(\R^n)}
  &\les \|(1+|\cdot|)^{-1+\frac{1}{p}-\frac{1}{q}} \ast \|f(\cdot,z)\|_{L^p(\R^{n-1})}\|_{L^q(\R)} \\
  &\les \|(1+|\cdot|)^{-1+\frac{1}{p}-\frac{1}{q}}\|_{L^{\frac{pq}{pq-q+p},w}(\R)} 
  \left(\int_\R \|f(\cdot,z)\|_{L^p(\R^{n-1})}^p\:\mathrm{d}z\right)^{\frac{1}{p}}  \\
  &\les \|f\|_{L^p(\R^n)}. 
\end{align*}
\qed

\section{Proof of Proposition~\ref{prop:largefreq}} \label{sec:largefreq}

We recall that have to prove the following: For all $n\in\N,n\geq 2$ and $p', q \in [2,\infty]$ we have  
$$
  \|W\|_{L^q(\R^n)} + \sup_{0<|\eps|\leq 1}  \norm{W_{\eps}}_{L^q(\R^n)} \les \norm{f}_{L^p(\R^n)}.
$$
provided that $0 \leq \frac{1}{p} - \frac{1}{q} \leq \frac{2}{n}$
and $\frac{1}{p} - \frac{1}{q} < \frac{2}{n}$ if $p=1$ or $q=\infty$.
Here, 
\begin{align*}
    W(x,y)&= \cF_{n-1}^{-1}\left(
    \e{\i |y|\nu_{1}} \ind_{|\cdot|> \mu_1 + \mu_2}
    (m_1 g_++  m_2 g_-)\right)(x) \\
    &\;+ \cF_{n-1}^{-1}\left(\e{\i |y|\nu_{2}} \ind_{|\cdot|> \mu_1 + \mu_2}
     (m_3 g_+ +   m_4 g_-) \right)(x).
\end{align*}   
and $W_\eps$ is given by the same formula with $\nu_1,\nu_2$ replaced by $\nu_{1,\eps},\nu_{2,\eps}$,
respectively. We recall $g_+(\xi)=\cF_n^+f(\xi,-\nu_1(\xi))$ and $g_-(\xi)=\cF_n^-(\xi,\nu_2(\xi))$.

\medskip

\noindent\textbf{Proof of Proposition~\ref{prop:largefreq}:}
  Again we concentrate on the estimates for $W$ since the corresponding modifications for
  $W_{\eps}$ are purely technical. We recall that for $|\xi|\geq R:=\mu_1+\mu_2$ 
  we have $|m_1(\xi)|+\ldots+|m_4(\xi)|\les (1+|\xi|)^{-1}$ as well as $\i\nu_j(\xi)  = -\sqrt{|\xi|^2 -
  \mu_j^2} \leq - c(|\xi|+1)$ for some $c>0$, see~\eqref{eq:mj_estimates} and~\eqref{eq:nu}.  
  So the Hausdorff-Young inequality implies
 \begin{align*}
   \norm{W(\cdot,y)}_{L^q(\R^{n-1})}
   &\les  \left\| \e{\i |y|\nu_{1}}(m_1 g_+ + m_2 g_-) 
     + \e{\i |y|\nu_{2}} (m_3 g_+ + m_4 g_-) \right\|_{L^{q'}(\R^{n-1}\sm B_R(0))} \\
   &\les    \left\|
    \e{-c|y|(|\cdot|+1)} ( |\cdot|+1)^{-1} (|g_+| + |g_-|) 
    \right\|_{L^{q'}(\R^{n-1}\sm B_R(0))}  \\
   %&\les  \left(\int_\R  \left\|
   % \e{-c|y|(|\cdot|+1)} (1+|\cdot|)^{-1} \int_\R |\cF_{n-1}[f(\cdot,z)](\xi)| 
   % \e{-c|z| (|\cdot| + 1)}\dz     
   % \right\|_{L^{q'}(\R^{n-1})}^q\dy\right)^{\frac{1}{q}} \\
   &\les  \int_\R  \left\|
    \e{-c(|y|+|z|)(|\cdot|+1)} (|\cdot|+1)^{-1} \cF_{n-1}[f(\cdot,z)]  
    \right\|_{L^{q'}(\R^{n-1})} \dz.     
%         &\les \left(\int_\R 
%         \norm{1_{|\cdot|\geq R} \e{-|y|\sqrt{|\xi|^2-\mu_j^2}}   m_{j, y}(\cdot) }_{L^{q'}(\R^{n-1})}^q 
%         \, \mathrm{d}y \right)^\frac{1}{q} 
%         \\
%         &\les \left(\int_\R \left(\int_{|\xi|\geq R} 
%         | \e{-c |y|(|\xi| + 1)}   m_{j, y}(\xi)|^{q'} 
%         \, \mathrm{d}\xi\right)^\frac{q}{q'} 
%         \, \mathrm{d}y \right)^\frac{1}{q}.
%         \\
%         \intertext{Inserting now the estimate from above we get}
%         &\les \left(\int_\R \left(\int_{|\xi|\geq R} 
%         \left[\int_\R  
%         \left|\cF_{n-1} [ f (\, \cdot \, , z)] (\xi)\right| 
%         \frac{\e{- c (|y| + |z|) (|\xi| + 1)}}{|\xi|+1} 
%         \,\mathrm{d}z \right]^{q'} 
%         \, \mathrm{d}\xi\right)^\frac{q}{q'} 
%         \, \mathrm{d}y \right)^\frac{1}{q}
%         \\
%         \intertext{and obtain using Minkowski's inequality in integral form}
%         &\les \left(\int_\R 
%         \left(  \int_\R  
%         \left(  \int_{|\xi|\geq R}
%         \left|\cF_{n-1} [ f (\, \cdot \, , z)] (\xi)\right|^{q'} 
%         \frac{\e{- q' c (|y| + |z|) (|\xi| + 1)}}{(|\xi|+1)^{q'}} 
%         \, \mathrm{d}\xi\right)^\frac{1}{q'}
%         \,\mathrm{d}z \right)^{q} 
%         \, \mathrm{d}y \right)^\frac{1}{q}.
%         \\
    \intertext{In the last step we applied Minkowski's inequality in integral form. Using now H\"older's and
    then the Hausdorff-Young inequality we get} 
    \|W\|_{L^q(\R^n)}
    &  \leq
	\left(  
		\int_\R
		\left(  
		\int_\R		
		\norm{\frac{\e{- c (|z| + |y|) (|\cdot|+ 1)}}{|\,\cdot\,| + 1}}_{L^{\frac{pq}{q-p}}(\R^{n-1})}
		\: \norm{\cF_{n-1} [f(\,\cdot\, , z)]}_{L^{p'}(\R^{n-1})}
		\: \mathrm{d} z
		\right)^q
		\: \mathrm{d}y
	\right)^\frac{1}{q}
	\\
	&  \les
	\left(  
		\int_\R
		\left(  
		\int_\R		
			\norm{\frac{\e{- c  (|z| + |y|) (|\cdot|+ 1)}}{|\,\cdot\,| + 1}}_{L^{\frac{pq}{q-p}}(\R^{n-1})}
			\: \norm{f(\,\cdot\, , z)}_{L^{p}(\R^{n-1})}
		\: \mathrm{d} z
		\right)^q
		\: \mathrm{d}y
	\right)^\frac{1}{q}
	\\
% 	&  \les
% 	\left(  
% 		\int_\R
% 		\left(  
% 		\int_\R		
% 			\norm{\frac{\e{- c |y - z| (|\cdot|+ 1)}}{|\,\cdot\,| + 1}}_{L^{\frac{pq}{q-p}}(\R^{n-1})}
% 			\: \norm{f(\,\cdot\, , z)}_{L^{p}(\R^{n-1})}
% 		\: \mathrm{d} z
% 		\right)^q
% 		\: \mathrm{d}y
% 	\right)^\frac{1}{q}
% 	\\
	&  \les \norm{K \ast F}_{L^q(\R)}
\end{align*}     
where
\begin{align*}
	K(w) =
	\norm{(|\cdot| + 1)^{-1} \:
			\e{- c |w| (|\,\cdot\,| + 1)} }_{L^{\frac{pq}{q-p}}(\R^{n-1})}
		\quad \text{and} \quad
		F(w) = \norm{f(\,\cdot\, , w)}_{L^p(\R^{n-1})}.
\end{align*}
For $w > 0$ we have
\begin{align*}
	K(w) 
	&\eqsim
	\e{-c  w} \,  \left( 
		\int_0^1
			r^{n-2}  
		\: \mathrm{d}r
		+
		\int_1^\infty
			r^{n-2 - \frac{pq}{q-p}}  
			\: \: \e{-c \frac{pq}{q-p} w r}
		\: \mathrm{d}r
		\right)^\frac{q-p}{pq}
		\\
		&=
		\e{-c  w} \,  \left( 
		\int_0^1
			r^{n-2}  
		\: \mathrm{d}r
		+
		w^{-(n-1) + \frac{pq}{q-p}} \, 
		\int_w^\infty
			t^{n-2 - \frac{pq}{q-p}}  
			\: \: \e{-c \frac{pq}{q-p} t}
		\: \mathrm{d}t
		\right)^\frac{q-p}{pq}
		\\
		&\eqsim
		1_{[1, \infty)}(w) \e{-c  w} +  1_{(0,1)}(w) w^{-(n-1)\frac{q-p}{pq} + 1}.
\end{align*}
Thus $K \in L^{pq/(pq+p-q), \infty}(\R)$ if and only if
\begin{align*}
	\left( -(n-1)\frac{q-p}{pq} + 1\right) \cdot \frac{pq}{pq+p-q} \geq -1,
	\quad \text{equivalently} \quad
	\frac{1}{p} - \frac{1}{q} \leq \frac{2}{n},
\end{align*}
which is precisely the assumption in the proposition. Similarly, $K \in L^{pq/(pq+p-q)}(\R)$ if and only if
$\frac{1}{p} - \frac{1}{q}< \frac{2}{n}$.
So, as in the proof of
Proposition~\ref{prop:intermediatefreq}, the classical and weak-space versions of Young's convolution
inequality imply $\norm{W}_{L^q(\R^n)}\les \norm{f}_{L^p(\R^n)}$, which is all we had to show. 
\qed

% \begin{itemize}
%   \item  In order to show that these exponents are best possible define $g\in L^2(A)$
%   via $g(\xi):= 1_A(\xi)(|\xi|^2-1)^{-a/2}$ where $a<1$. For $\delta:= \frac{a+\alpha}{2}\in
%   (0,\frac{1+\alpha}{2})\subset (0,1)$ we obtain
%   \begin{align*}
%     (S_{0,\alpha}^*g)(x) 
%     &= \int_1^2 r^{d-1}(r^2-1)^{-(a+\alpha)/2} \int_{S^{d-1}}  e^{irx\cdot\omega}\:\mathrm{d}\sigma(\omega) \\
%     &= \int_1^2 r^{d-1}(r^2-1)^{-(a+\alpha)/2} |S^{d-1}| (r|x|)^{\frac{2-d}{2}}  J_{\frac{d-2}{2}}(r|x|) \\
%     &\sim |x|^{\frac{2-d}{2}}  \int_1^2 (r-1)^{-(a+\alpha)/2}    J_{\frac{d-2}{2}}(r|x|) \\
%     &\sim |x|^{\delta - \frac{1+d}{2}} (1+O(|x|^{-\delta})) 
%   \end{align*}
%   and hence $S_{0,\alpha}^*\in L^{p'}(\R^d)$ precisely for $p'>\frac{2d}{d+1-a-\alpha}$. 
%   We provide a rigorous justification for the above-mentioned asymptotics in \r{Appendix ???}. 
%   So we conclude that $p\leq \frac{2d}{d+\alpha}$ is a necessary condition and that our bounds might
%   only be improved by a result in the endpoint case $p=\frac{2d}{d+\alpha}$. \\  
%   \r{Discuss endpoint case! Was passiert fuer $g(\xi):=
%   1_A(\xi)(|\xi|^2-1)^{-(2-\alpha)/2}|\log(|\xi|^2-1)|^{-\tau}$?}
% \end{itemize}

\section*{Acknowledgments}
 
Funded by the Deutsche Forschungsgemeinschaft (DFG, German Research Foundation) - Project-ID 258734477 - SFB
1173.

 \section*{Conflict of interest}
 
 The authors declare that they have no conflict of interest.

\bibliographystyle{plain}
\bibliography{doc}

\end{document}